\documentclass{theoretics}

\usepackage{mathtools,amsfonts}
\usepackage[capitalise,noabbrev,nameinlink]{cleveref}

\usepackage{tikz}
\usetikzlibrary{arrows,backgrounds,positioning,fit,cd}
\allowdisplaybreaks
\usepackage[font=normalsize]{subcaption}

\title{Lasserre Hierarchy for Graph Isomorphism and Homomorphism Indistinguishability}

\ThCSshorttitle{Lasserre Hierarchy for Graph Isomorphism and Homomorphism Indistinguishability}

 \ThCSshortnames{D.~E.~Roberson, T.~Seppelt}

\ThCSauthor[1,2]{David E.\  Roberson}{dero@dtu.dk}[https://orcid.org/0000-0002-4463-8095]
\ThCSauthor[3]{Tim Seppelt}{seppelt@cs.rwth-aachen.de}[https://orcid.org/0000-0002-6447-0568]

\ThCSaffil[1]{Department of Applied Mathematics and Computer Science, 
     Technical University of Denmark}
\ThCSaffil[2]{QMATH, Department of Mathematical Sciences, 
     University of Copenhagen, Denmark} 
\ThCSaffil[3]{RWTH Aachen University, Germany}

\ThCSthanks{This work was presented at the 50th International Colloquium on Automata, Languages, and Programming (ICALP 2023)~\cite{icalp_version}.    
Initial discussions took place at Dagstuhl Seminar 22051 ``Finite and Algorithmic Model Theory''. The first author is supported by the Carlsberg
    Foundation Young Researcher Fellowship CF21-0682 -- ``Quantum Graph Theory''. The second author is supported by the Deutsche Forschungsgemeinschaft (DFG, German Research Foundation) via Research Training Group 2236/2 UnRAVeL and the European Union (ERC, SymSim, 101054974).
    }

\ThCSkeywords{Lasserre hierarchy, homomorphism indistinguishability, Sherali--Adams hierarchy, treewidth, semidefinite programming, linear programming, graph isomorphism}

\ThCSyear{2024}
\ThCSarticlenum{20}
\ThCSreceived{Sep 25, 2023}
\ThCSrevised{Jun 10, 2024}
\ThCSaccepted{Jul 7, 2024}
\ThCSpublished{Sep 2, 2024}
\ThCSdoicreatedtrue

\ThCSnewtheoita{question}

\Crefname{claim}{Claim}{Claims}
\Crefname{fact}{Fact}{Facts}
\Crefname{observation}{Observation}{Observations}

\newcommand{\abs}[1]{\left| #1 \right|}

\newcommand{\down}[1]{\soe(#1)} 
\renewcommand\phi\varphi
\renewcommand\epsilon\varepsilon

\DeclareMathOperator{\soe}{soe}
\DeclareMathOperator{\pw}{pw}
\DeclareMathOperator{\rel}{atp}
\DeclareMathOperator{\id}{id}
\DeclareMathOperator{\tr}{tr}
\DeclareMathOperator{\tw}{tw}
\DeclareMathOperator{\ISO}{ISO}

\newcommand\multiset[1]{\left\{ \!\! \left\{ #1 \right\} \!\! \right\}}
\begin{document}
	\maketitle

	\begin{abstract}
    We show that feasibility of the $t^\text{th}$ level of the Lasserre semidefinite programming hierarchy for graph isomorphism can be expressed as a homomorphism indistinguishability relation. In other words, we define a class $\mathcal{L}_t$ of graphs such that graphs $G$ and $H$ are not distinguished by the $t^\text{th}$ level of the Lasserre hierarchy if and only if they admit the same number of homomorphisms from any graph in $\mathcal{L}_t$.  By analysing the treewidth of graphs in $\mathcal{L}_t$, we prove that the $3t^\text{th}$ level of Sherali--Adams linear programming hierarchy is as strong as the $t^\text{th}$ level of Lasserre. Moreover, we show that this is best possible in the sense that $3t$ cannot be lowered to $3t-1$ for any $t$. The same result holds for the Lasserre hierarchy with non-negativity constraints, which we similarly characterise in terms of homomorphism indistinguishability over a family $\mathcal{L}_t^+$ of graphs.  Additionally, we give characterisations of level-$t$ Lasserre with non-negativity constraints in terms of logical equivalence and via a graph colouring algorithm akin to the Weisfeiler--Leman algorithm. This provides a polynomial time algorithm for determining if two given graphs are distinguished by the $t^\text{th}$ level of the Lasserre hierarchy with non-negativity constraints.
 \end{abstract}

    \section{Introduction}

    The aim of this paper is to relate two rich sets of tools used to distinguish non-isomorphic graphs: the Lasserre semidefinite programming hierarchy and homomorphism indistinguishability.
    
    Distinguishing non-isomorphic graphs is a ubiquitous problem in the theoretical and practical study of graphs. The ability of certain graph invariants to distinguish graphs has long been a rich area of study, leading to fundamental questions such as the long-standing open problem of whether almost all graphs are determined by their spectrum \cite{van_dam_which_2003}. 
    In practice, deploying e.g.\  machine learning architectures powerful enough to distinguish graphs with different features is of great importance \cite{grohe_word2vec_2020}. 
    This motivates an in-depth study of the power of various graph invariants and tools used to distinguish graphs.

    Among such techniques is the Lasserre semidefinite programming hierarchy \cite{lasserre_global_2001} which can be used to relax the  integer quadratic program for graph isomorphism $\ISO(G, H)$, cf.\  \cref{sec:systems}.
    This yields a sequence of semidefinite programs, i.e.\  the level-$t$ Lasserre relaxation of $\ISO(G, H)$ for $t\geq 1$, which are infeasible for more and more non-isomorphic graphs as $t$ grows. 
    In~\cite{snook_graph_2014,odonnell_hardness_2014,berkholz_limitations_2015}, it was shown that in general only the level-$\Omega(n)$ Lasserre system of equations can distinguish all non-isomorphic $n$-vertex graphs. 
    In~\cite{atserias_definable_2023}, the Lasserre hierarchy was compared with the Sherali--Adams linear programming hierarchy \cite{sherali_hierarchy_1990}, which is closely related to the Weisfeiler--Leman algorithm \cite{weisfeiler_construction_1976,atserias_sheraliadams_2013,grohe_pebble_2015,malkin_sheraliadams_2014}, the arguably most relevant combinatorial method for distinguishing graphs. 
    In general and not only for graph isomorphism, feasibility of the level-$t$ Lasserre relaxation of an integer program implies feasibility of its level-$t$ Sherali--Adams relaxation \cite{laurent_comparison_2003}.
    For graph isomorphism, it was shown in \cite{atserias_definable_2023} that the converse holds up to multiplicative offset in the number of levels.
    Thus, perhaps surprisingly, the Lasserre hierarchy is not more powerful than the Sherali--Adams hierarchies when applied to $\ISO(G, H)$.
    More precisely, by \cite[Corollary~6.7]{atserias_definable_2023},
    there exists a constant $c$ such that if the level-$ct$ Sherali--Adams relaxation of $\ISO(G, H)$  is feasible for two graphs $G$ and $H$ then so is the level-$t$ Lasserre relaxation.
    
    Another set of expressive equivalence relations comparing graphs is given by homomorphism indistinguishability, a notion originating from the study of graph substructure counts.
    Two graphs $G$ and $H$ are \emph{homomorphism indistinguishable} over a family of graphs $\mathcal{F}$, in symbols $G \equiv_{\mathcal{F}} H$, if the number of homomorphisms from $F$ to $G$ is equal to the number of homomorphisms from $F$ to $H$ for every graph $F \in \mathcal{F}$. The study of this notion began in 1967, when Lovász~\cite{Lovasz67} showed that two graphs $G$ and $H$ are isomorphic if and only if they are homomorphism indistinguishable over all graphs.
    In recent years, many prominent equivalence relations comparing graphs were characterised as homomorphism indistinguishability relations over restricted graph classes \cite{dell_lovasz_2018,dvorak_recognizing_2010,grohe_counting_2020,dawar_lovasz_2021,mancinska_quantum_2019,grohe_homomorphism_2022,atserias_expressive_2021,montacute_pebble_2022,abramsky_discrete_2022,roberson_oddomorphisms_2022,rattan_weisfeiler_2023}. For example, a folklore result asserts that two graphs have cospectral adjacency matrices iff they are homomorphism indistinguishable over all cycle graphs, cf.\  \cite{grohe_homomorphism_2022}. Two graphs are quantum isomorphic iff they are homomorphism indistinguishable over all planar graphs \cite{mancinska_quantum_2019}. 
    Furthermore, feasibility of the level-$t$ Sherali--Adams relaxation of $\ISO(G, H)$ has been characterised as homomorphism indistinguishability over all graphs of treewidth at most $t-1$ \cite{atserias_sheraliadams_2013,grohe_pebble_2015,malkin_sheraliadams_2014}. 
    In this way, notions from logic \cite{dvorak_recognizing_2010,grohe_counting_2020,rattan_weisfeiler_2023}, category theory \cite{dawar_lovasz_2021,montacute_pebble_2022,abramsky_discrete_2022}, algebraic graph theory \cite{dell_lovasz_2018,grohe_homomorphism_2022}, and quantum groups \cite{mancinska_quantum_2019} have been related to homomorphism indistinguishability.

    \subsection{Contributions}

    Although feasibility of the level-$t$ Lasserre relaxation of $\ISO(G, H)$ was sandwiched between feasibility of the level-$ct$ and level-$t$ Sherali--Adams relaxation in \cite{atserias_definable_2023},  
    the constant $c$ remained unknown. 
    In fact, this $c$ is not explicit and depends on the implementation details of an algorithm developed in that paper.
	Our main result asserts that $c$ can be taken to be three and that this constant is best possible.\footnote{The constant $c$ in \cite[Theorem~6.3]{atserias_definable_2023} depends on the implementation details of the algorithm that yields their Corollary~5.1. This algorithm is also dependent on the precise version of the Lasserre system of equations used there. 
    As discussed in \cref{sec:systems,app:lasserre}, our Lasserre system of equations is defined slightly differently.
    In \cref{thm:main}, we abstract from these details by proving a statement that involves only the equivalence relations $\simeq_t^{\textup{SA}}$ and $\simeq_t^{\textup{L}}$. 
    Since our Lasserre formulation and the one in \cite{atserias_definable_2023} are equivalent (\cref{lem:systems}), \cref{thm:main} yields that $c$ in \cite[Theorem~6.3]{atserias_definable_2023} can be taken to be three (and that this is best possible). Our results do not imply bounds on the complexity of the algorithm yielding \cite[Corollary~5.1]{atserias_definable_2023}.}

	\begin{theorem} \label{thm:main}
		For two graphs $G$ and $H$ and every $t \geq 1$, the following implications hold:
		\[
		G \simeq_{3t}^{\textup{SA}} H 
		\implies
		G \simeq_t^{\textup{L}} H
		\implies
		G \simeq_{t}^{\textup{SA}} H
		\]
        Furthermore, for every $t \geq 1$, there exist graphs $G$ and $H$ such that $G \simeq_{3t-1}^{\textup{SA}} H$ and $G \not\simeq_t^{\textup{L}} H$.
	\end{theorem}
    Here,  $G \simeq_t^{\textup{L}} H$ and $G \simeq_t^{\textup{SA}} H$ denote that the level-$t$ Lasserre relaxation and respectively the level-$t$ Sherali--Adams relaxation of $\ISO(G, H)$ are feasible.
    
	\Cref{thm:main} is proven using the framework of homomorphism indistinguishability. 
    In previous works \cite{dell_lovasz_2018,mancinska_graph_2020,grohe_homomorphism_2022,rattan_weisfeiler_2023}, 
    the feasibility of various systems of equations associated to graphs like the Sherali--Adams relaxation of $\ISO(G, H)$ was characterised in terms of homomorphism indistinguishability over certain graph classes. 
    We continue this line of research by characterising the feasibility of the level-$t$ Lasserre relaxation of $\ISO(G, H)$ by homomorphism indistinguishability of $G$ and $H$ over the novel class of graphs~$\mathcal{L}_t$
    introduced in \cref{def:l-lplus}.
	
	\begin{theorem} \label{thm:main2}
		For every integer $t \geq 1$, there is a minor-closed graph class $\mathcal{L}_t$ of graphs of treewidth at most $3t-1$ such that for all graphs $G$ and $H$ it holds that $G \simeq_t^{\textup{L}} H $ if and only if $G \equiv_{\mathcal{L}_t} H$.
	\end{theorem}
 
    The bound on the treewidth of graphs in $\mathcal{L}_t$ in \cref{thm:main2} yields the upper bound in \cref{thm:main} given the result of \cite{atserias_sheraliadams_2013,grohe_pebble_2015,atserias_definable_2023,dvorak_recognizing_2010} that two graphs $G$ and $H$ satisfy $G \simeq_t^{\text{SA}} H$ if and only if they are homomorphism indistinguishable over the class $\mathcal{TW}_{t-1}$ of graphs of treewidth at most $t-1$.
    To our knowledge, \cref{thm:main} is the first result which tightly relates equivalence relations on graphs by comparing the graph classes which characterise them in terms of homomorphism indistinguishability.

    Our techniques extend to a stronger version of the Lasserre hierarchy which imposes non-negativity constraints on all variables. 
    Denoting feasibility of the level-$t$ Lasserre relaxation of $\ISO(G, H)$ with non-negativity constraints by $G \simeq_t^{\textup{L}^+} H$, 
    we characterise $\simeq_t^{\textup{L}^+}$ in terms of homomorphism indistinguishability over the graph class $\mathcal{L}_t^+$, defined in \cref{def:l-lplus} as a super class of $\mathcal{L}_t$.
    This is in line with previous work in \cite{dell_lovasz_2018,grohe_homomorphism_2022}, where the feasibility of the level-$t$ Sherali--Adams relaxation of $\ISO(G, H)$ without non-negativity constraints was characterised as homomorphism indistinguishable over the class $\mathcal{PW}_{t-1}$ of graphs of pathwidth at most $t-1$. 
    
	\begin{theorem} \label{thm:main3}
		For every integer $t \geq 1$, there is a minor-closed graph class $\mathcal{L}_t^+$ of graphs of treewidth at most $3t-1$ such that for all graphs $G$ and $H$ it holds that $G \simeq_t^{\textup{L}^+} H $ if and only if $G \equiv_{\mathcal{L}_t^+} H$.
	\end{theorem}

    Given the aforementioned correspondence between the Sherali--Adams relaxation with and without non-negativity constraints and homomorphism indistinguishability over graphs of bounded treewidth and pathwidth, we conduct a detailed study of the relationship between the class of graphs of bounded treewidth, pathwidth, and the classes $\mathcal{L}_t$ and $\mathcal{L}_t^+$. Their results, depicted in \cref{fig:relationship}, yield independent proofs of the known relations between feasibility of the Lasserre relaxation with and without non-negativity constraints and the Sherali--Adams relaxation with and without non-negativity constraints \cite{berkholz_limitations_2015,atserias_definable_2023,grohe_homomorphism_2022} using the framework of homomorphism indistinguishability.

    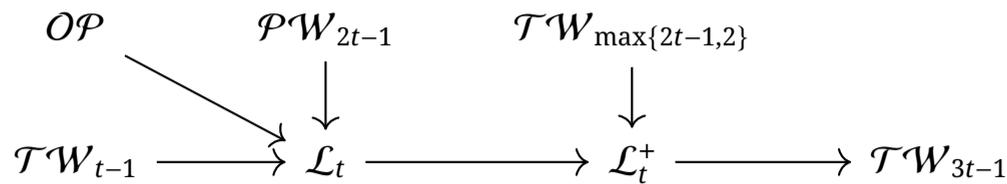
\begin{figure}
        \centering
		\begin{tikzcd}
			\mathcal{OP} \arrow[dr] & \mathcal{PW}_{2t-1} \arrow[d]                        & \mathcal{TW}_{\max\{2t-1, 2\}} \arrow[d] &                     \\
			\mathcal{TW}_{t-1} \arrow[r] & \mathcal{L}_t \arrow[r] & \mathcal{L}_t^+ \arrow[r]     & \mathcal{TW}_{3t-1}
		\end{tikzcd}
		\caption{Relationship between $\mathcal{L}_t$, $\mathcal{L}_t^+$, the classes of graphs of bounded treewidth, bounded pathwidth, and the class of outerplanar graphs.	
			An arrow $\mathcal{A} \rightarrow \mathcal{B}$ indicates that $\mathcal{A} \subseteq \mathcal{B}$ and thus that $G \equiv_{\mathcal{B}} H$ implies $G \equiv_{\mathcal{A}} H$ for all graphs $G$ and $H$. For formal statements, see \cref{sec:lt-ltplus,sec:t-equals-one}.}
		\label{fig:relationship}
	\end{figure}

    In the course of proving \cref{thm:main2,thm:main3}, we derive further equivalent characterisations of $\simeq_t^{\textup{L}}$ and $\simeq_t^{\textup{L}^+}$. These characterisations, which are mostly of a linear algebraic nature, ultimately yield a characterisation of $\simeq_t^{\textup{L}^+}$ in terms of a fragment of first-order logic with counting quantifiers and indistinguishability under a polynomial time algorithm akin to the Weisfeiler--Leman algorithm.
    In this way, we obtain the following algorithmic result.
    It implies that \emph{exact} feasibility of the Lasserre semidefinite program with non-negativity constraints can be tested in polynomial time.
    In general, only the \emph{approximate} feasibility of semidefinite programs can be decided efficiently, e.g.\  using the ellipsoid method \cite{grotschel_geometric_2012,atserias_definable_2023}.
    Our reformulations of $\simeq_t^{\textup{L}}$ fall short of yielding a polynomial-time algorithm for exact feasibility of the Lasserre semidefinite program without non-negativity constraints.
    
    \begin{theorem} \label{thm:main5}
		Let $t \geq 1$.
		Given graphs $G$ and $H$, it can be decided in polynomial time whether $G \simeq_t^{\textup{L}^+} H $.
	\end{theorem}
	
	Finally, for $t = 1$, we show that $\mathcal{L}_1$ and $\mathcal{L}_1^+$ are respectively equal to the class $\mathcal{OP}$ of outerplanar graphs and to the class of graphs of treewidth at most $2$. The following \cref{thm:main4} parallels a result of \cite{mancinska_quantum_2019} asserting that two graphs $G$ and $H$ are indistinguishable under the $2$-WL algorithm iff $G \simeq_1^{\textup{L}^+} H$.
	
	\begin{theorem} \label{thm:main4}
		Two graphs $G$ and $H$ satisfy $G \simeq_1^{\textup{L}} H$ if and only if $G \equiv_{\mathcal{OP}} H$.
	\end{theorem}
	
    \subsection{Techniques}

    In the first part of the paper (\cref{sec:lasserre-to-tensor}), linear algebraic tools developed in \cite{mancinska_relaxation_2017,mancinska_quantum_2019} are generalised to yield reformulations of the entire Lasserre hierarchy with and without non-negativity constraints. 
    \Cref{sec:hi} is concerned with the graph theoretic properties of the graph classes $\mathcal{L}_t$ and $\mathcal{L}_t^+$.
    For understanding the homomorphism indistinguishability relations over these graph classes, 
    the framework of bilabelled graphs and their homomorphism tensors developed in \cite{mancinska_graph_2020,grohe_homomorphism_2022} is used.
    Despite this, our approach is different from \cite{grohe_homomorphism_2022,rattan_weisfeiler_2023} in the sense that here the graph classes $\mathcal{L}_t$ and $\mathcal{L}_t^+$ are inferred from given systems of equations, namely the Lasserre relaxation, rather than that a system of equations is built for a given graph class.

	\section{Preliminaries}
	
	\subsection{Linear Algebra}
	
	Let $\mathcal{PSD}$ denote the family of real \emph{positive semidefinite matrices}, i.e.\  of matrices $M$ of the form $M_{ij} = v_i^Tv_j$ for vectors $v_1, \dots, v_n$, the \emph{Gram vectors} of $M$. Write $M \succeq 0$ iff $M \in \mathcal{PSD}$.
	Let $\mathcal{DNN}$ denote the family of \emph{doubly non-negative matrices}, i.e.\  of entry-wise non-negative positive semidefinite matrices.

    Let $n,m \geq 1$. Write $\id_n \in \mathbb{C}^{n \times n}$ for the identity matrix.
    The \emph{tensor product} of two matrices $X = (x_{ij}) \in \mathbb{C}^{n\times n}$ and $Y \in \mathbb{C}^{m \times m}$ is the block matrix \[ X \otimes Y = \left( \begin{matrix} x_{11}Y & \dots & x_{1n} Y \\ \vdots & \ddots & \vdots \\ x_{n1}Y & \dots & x_{nn} Y \end{matrix} \right) \in \mathbb{C}^{nm \times nm}. \]

    A \emph{tensor} is an element $A \in \mathbb{C}^{n^t \times n^t}$ for some $n, t \in \mathbb{N}$.
    For a tensor $A \in \mathbb{C}^{n^t \times n^t}$,
    write $\soe(A) \coloneqq \sum_{i,j =1}^{n^t} A_{ij}$ for its \emph{sum-of-entries}.
    The symmetric group $\mathfrak{S}_{2t}$ acts on $\mathbb{C}^{n^t \times n^t}$ by permuting the coordinates, i.e.\ 
    for all $\boldsymbol{u},\boldsymbol{v} \in [n]^t$ and $\sigma \in \mathfrak{S}_{2t}$,
    $A^\sigma(\boldsymbol{u},\boldsymbol{v}) \coloneqq A(\boldsymbol{x}, \boldsymbol{y})$ where $\boldsymbol{x}_i \coloneqq (\boldsymbol{uv})_{\sigma^{-1}(i)}$ and  $\boldsymbol{y}_{j-t} \coloneqq (\boldsymbol{uv})_{\sigma^{-1}(j)}$ for all $1 \leq i \leq t < j \leq 2t$.

    We recall the following lemmas from \cite{mancinska_graph_2020}.
    A linear map $\Phi \colon \mathbb{C}^{m\times m} \to \mathbb{C}^{n \times n}$ is
		 \emph{trace-preserving} if $\tr(\Phi(X)) =\tr(X)$ for all $X \in \mathbb{C}^{m\times m}$,
		 \emph{unital} if $\Phi(\id_m) = \id_n$,
		 \emph{$\mathcal{K}$-preserving} for a family of matrices $\mathcal{K}$ if $\Phi(K) \in \mathcal{K}$ for all $K \in \mathcal{K}$,
		 \emph{positive} if it is $\mathcal{PSD}$-preserving, i.e.\  if $\Phi(X)$ is positive semidefinite for all positive semidefinite $X$,
		 \emph{completely positive} if $\id_r \otimes \Phi$ is positive for all $r \in \mathbb{N}$.
	The \emph{Choi matrix} of $\Phi$ is $C_\Phi = \sum_{i,j=1}^m E_{ij} \otimes \Phi(E_{ij}) \in \mathbb{C}^{mn \times mn}$. 
 Here, $E_{ij} \in \mathbb{C}^{m \times m}$ denotes the matrix whose $(i,j)$-th entry is $1$ and all whose other entries are zero.
    The statement of \cref{lem:mg-4.4} for $\mathcal{PSD}$ is well-known, cf.\ e.g.\ \cite{choi_linear_1975}.
    
    \begin{lemma}[{\cite[Lemma~4.4]{mancinska_graph_2020}}] \label{lem:mg-4.4}
		Consider a family of matrices $\mathcal{K} \in \{\mathcal{DNN}, \mathcal{PSD}\}$ and a linear map $\Phi \colon \mathbb{C}^{m\times m} \to \mathbb{C}^{n\times n}$.
		The following are equivalent:
		\begin{enumerate}
			\item the map $\id_m \otimes \Phi$ is $\mathcal{K}$-preserving,
			\item the Choi matrix $C_\Phi$ lies in $\mathcal{K}$,
			\item $\Phi$ is \emph{completely $\mathcal{K}$-preserving}, i.e.\  $\id_r \otimes \Phi$ is $\mathcal{K}$-preserving for all $r \in \mathbb{N}.$
		\end{enumerate}
	\end{lemma}

    For $\Phi \colon \mathbb{C}^{m\times m} \to \mathbb{C}^{n\times n}$, write $\Phi^* \colon \mathbb{C}^{n\times n} \to \mathbb{C}^{m \times m}$ for the \emph{adjoint} of $\Phi$.
    As a matrix, $\Phi^*$ is the conjugate transpose of $\Phi$.
	
	\begin{lemma}[{\cite[Lemma~4.10]{mancinska_graph_2020}}] \label{lem:mg-4.9}
		Let $\Phi \colon \mathbb{C}^{n\times n} \to \mathbb{C}^{n\times n}$ be a linear map which is completely positive, trace-preserving, and unital.
		Then for any matrix $X$ such that $\Phi^*(\Phi(X)) = X$ it holds that
		$\Phi(XW) = \Phi(X)\Phi(W)$ and $\Phi(WX) = \Phi(W)\Phi(X)$ for all $W \in \mathbb{C}^{n\times n}$.
	\end{lemma}

    A vector space $\mathcal{A} \subseteq \mathbb{C}^{n\times n}$ is an \emph{algebra} if it is closed under matrix multiplication.
    It is \emph{unital} if it contains the identity matrix $\id_n$.
    It is \emph{self-adjoint} if it is closed under taking conjugate transposes.
 
	\begin{lemma}[{\cite[Lemma~5.1]{mancinska_graph_2020}}] \label{lem:mg-A.1}
		Let $\mathcal{A}$ and $\mathcal{B}$ be self-adjoint unital subalgebras of $\mathbb{C}^{n\times n}$
		and $\phi \colon \mathcal{A} \to \mathcal{B}$ be a trace-preserving isomorphism such that $\phi(X^*) = \phi(X)^*$ for all $X \in \mathcal{A}$. Then there exists a unitary $U \in \mathbb{C}^{n\times n}$ such that $\phi(X) = UXU^*$ for all $X \in \mathcal{A}$.
	\end{lemma}

    For two vectors $v, w \in \mathbb{C}^n$, write $v \odot w$ for their \emph{Schur product}, i.e.\ $(v \odot w)(i) \coloneqq v(i)w(i)$ for all $i\in [n]$.
    
	\begin{lemma}[{\cite[Lemma~4.5]{mancinska_graph_2020}}] \label{lem:mg-4.5}
		Let $D \in \mathbb{C}^{m \times n}$ be a matrix and let $u \in \mathbb{C}^n$ and $v \in \mathbb{C}^m$. Then the following are equivalent:
		\begin{enumerate}
			\item $D(u \odot w) = v \odot (Dw)$ for all $w \in \mathbb{C}^n$,
			\item $D_{ij} = 0$  for all $i \in [m]$ and $j \in [n]$ such that $v_i \not= u_j$,
			\item $D^*(v \odot z) = u \odot (D^*z)$ for all $z \in \mathbb{C}^m$.
		\end{enumerate}
	\end{lemma}

    \subsection{Bilabelled Graphs and Homomorphism Tensors}
    \label{sec:bilabelled}
    All graphs in this article are undirected, finite, and without multiple edges. A graph is \emph{simple} if it does not contain any loops. 
    A \emph{homomorphism} $h \colon F \to G$ from a graph $F$ to a graph $G$ is a map $V(F) \to V(G)$ such that for all $uv \in E(F)$ it holds that $h(u)h(v) \in E(G)$. Note that this implies that any vertex in $F$ carrying a loop must be mapped to a vertex carrying a loop in $G$. 
    Write $\hom(F, G)$ for the number of homomorphisms from $F$ to $G$. For a family of graphs $\mathcal{F}$ and graphs $G$ and $H$ write $G \equiv_{\mathcal{F}} H$ if $G$ and $H$ are \emph{homomorphism indistinguishable over $\mathcal{F}$}, i.e.\  $\hom(F, G) = \hom(F, H)$ for all $F \in \mathcal{F}$.
    Since the graphs $G$ and $H$ into which homomorphisms are counted are throughout assumed to be simple, looped graphs in $\mathcal{F}$ can generally be disregarded as they do not admit any homomorphisms into simple graphs.

    We recall the following definitions from \cite{mancinska_quantum_2019,grohe_homomorphism_2022}.
    Let $k, \ell \geq 1$. A \emph{$(k, \ell)$-bilabelled graph} is a tuple $\boldsymbol{F} = (F, \boldsymbol{u}, \boldsymbol{v})$ where $F$ is a graph and $\boldsymbol{u} \in V(F)^k$, $\boldsymbol{v} \in V(F)^\ell$.
    The $\boldsymbol{u}$ are the \emph{in-labelled vertices} of $\boldsymbol{F}$ while the $\boldsymbol{v}$ are the \emph{out-labelled vertices} of $\boldsymbol{F}$. Given a graph $G$, the \emph{homomorphism tensor} of $\boldsymbol{F}$ for $G$ is $\boldsymbol{F}_G \in \mathbb{C}^{V(G)^k \times V(G)^\ell}$ whose $(\boldsymbol{x}, \boldsymbol{y})$-th entry is the number of homomorphisms $h \colon F \to G$ such that $h(\boldsymbol{u}_i) = \boldsymbol{x}_i$ and $h(\boldsymbol{v}_j) = \boldsymbol{y}_j$ for all $i \in [k]$ and $j \in [\ell]$.

    For a $(k, \ell)$-bilabelled graph $\boldsymbol{F} = (F, \boldsymbol{u}, \boldsymbol{v})$, write $\soe(\boldsymbol{F}) \coloneqq F$ for the underlying unlabelled graph of $\boldsymbol{F}$. 
    Here, $\soe$ stands for ``sum-of-entries''.
    If $k = \ell$, write $\tr(\boldsymbol{F})$ for the unlabelled graph underlying the graph obtained from $\boldsymbol{F}$ by identifying $\boldsymbol{u}_i$ with $\boldsymbol{v}_i$ for all $i\in [\ell]$.
    For $\sigma \in \mathfrak{S}_{k+\ell}$, write $\boldsymbol{F}^\sigma \coloneqq (F, \boldsymbol{x}, \boldsymbol{y})$ where $\boldsymbol{x}_i \coloneqq (\boldsymbol{uv})_{\sigma(i)}$ and  $\boldsymbol{y}_{j-k} \coloneqq (\boldsymbol{uv})_{\sigma(j)}$ for all $1 \leq i \leq k < j \leq k+\ell$, i.e.\  $\boldsymbol{F}^\sigma$ is obtained from $\boldsymbol{F}$ by permuting the labels according to $\sigma$.
    As a special case, define $\boldsymbol{F}^* \coloneqq (F, \boldsymbol{v}, \boldsymbol{u})$ the graph obtained by swapping in- and out-labels.
    
    Let $\boldsymbol{F} = (F, \boldsymbol{u}, \boldsymbol{v})$ and $\boldsymbol{F}' = (F', \boldsymbol{u}', \boldsymbol{v}')$ be $(k,\ell)$-bilabelled and $(m,n)$-bilabelled, respectively.
    If $\ell = m$, write $\boldsymbol{F} \cdot \boldsymbol{F}'$ for the $(k, n)$-bilabelled graph obtained from them by \emph{series composition}. 
    That is, the underlying unlabelled graph of $\boldsymbol{F} \cdot \boldsymbol{F}'$ is the graph obtained from the disjoint union of $F$ and $F'$ by identifying $\boldsymbol{v}_i$ and $\boldsymbol{u}'_i$ for all $i \in [\ell]$. 
    Multiple edges arising in this process are removed. 
    Loops are retained.
    The in-labels of $\boldsymbol{F} \cdot \boldsymbol{F}'$ lie on $\boldsymbol{u}$, the out-labels on $\boldsymbol{v}'$.
    Moreover, if $k = m$ and $\ell = n$, write $\boldsymbol{F} \odot \boldsymbol{F}'$ for the \emph{parallel composition} of $\boldsymbol{F}$ and $\boldsymbol{F}'$. 
    That is, the underlying unlabelled graph of the $(k, \ell)$-bilabelled graph $\boldsymbol{F} \odot \boldsymbol{F}'$ is the graph obtained from the disjoint union of $F$ and $F'$ by identifying $\boldsymbol{u}_i$ with $\boldsymbol{u}'_i$ and $\boldsymbol{v}_j$ with $\boldsymbol{v}'_j$  for all $i \in [k]$ and $j \in [\ell]$. 
    Again, multiple edges are dropped and loops retained. 
    The in-labels of $\boldsymbol{F} \odot \boldsymbol{F}'$ lie on $\boldsymbol{u}$, the out-labels on $\boldsymbol{v}$.
    
    As observed in \cite{mancinska_quantum_2019,grohe_homomorphism_2022}, the benefit of these combinatorial operations is that they have an algebraic counterpart. Formally, for all graphs $G$ and all $(\ell, \ell)$-bilabelled graphs $\boldsymbol{F}, \boldsymbol{F}'$, it holds that
    \begin{enumerate}
        \item $\soe(\boldsymbol{F}_G) = \hom(\soe(\boldsymbol{F}), G) $,
        \item $\tr(\boldsymbol{F}_G) = \hom(\tr(\boldsymbol{F}), G)$, 
        \item $(\boldsymbol{F}_G)^\sigma = (\boldsymbol{F}^\sigma)_G$, 
        \item $\boldsymbol{F}_G \cdot \boldsymbol{F}'_G = (\boldsymbol{F} \cdot \boldsymbol{F}')_G $, and
        \item $\boldsymbol{F}_G \odot \boldsymbol{F}'_G = (\boldsymbol{F} \odot \boldsymbol{F}')_G$.
    \end{enumerate}

    Slightly abusing notation, we say that two graphs $G$ and $H$ are homomorphism indistinguishable over a family of bilabelled graphs $\mathcal{S}$, in symbols $G \equiv_{\mathcal{S}} H$ if $G$ and $H$ are homomorphism indistinguishable over the family $\{\soe(\boldsymbol{S}) \mid \boldsymbol{S} \in \mathcal{S}\}$ of the underlying unlabelled graphs of the $\boldsymbol{S} \in \mathcal{S}$.

    \subsection{Pathwidth and Treewidth}
	
	\begin{definition} \label{def:definition}
		For graphs $F$ and $T$, a \emph{$T$-decomposition of $F$} is a map $\beta \colon V(T) \to 2^{V(F)}$ such that
		\begin{enumerate}
			\item $\bigcup_{t \in V(T)} \beta(t) = V(F)$,
			\item for every $e \in E(F)$, there is $t \in V(T)$ such that $e \subseteq \beta(t)$,
			\item for every $v \in V(F)$, the set of $t \in V(T)$ such that $v \in \beta(t)$ induces a connected subgraph of~$T$.
		\end{enumerate}
		The \emph{width} of a $T$-decomposition $\beta$ is $\max_{t\in V(T)} \abs{\beta(t)} - 1$.
	\end{definition}
    A $T$-decomposition for a tree $T$ is called a \emph{tree decomposition}.
    A $P$-decomposition for a path $P$ is called a \emph{path decomposition}.
	The \emph{treewidth} $\tw F$ of a graph $F$ is the minimal width of a tree decomposition. 
    Similarly, the \emph{pathwidth} $\pw F$  is the minimal width of a path decomposition. 
    For every $t \geq 0$, write $\mathcal{TW}_t$ and $\mathcal{PW}_t$ for the classes of all graphs of treewidth and respectively pathwidth at most $t$.
    The following slight generalisation of \cite[Lemma~8]{bodlaender_partial_1998} is used repeatedly in \cref{sec:lt-ltplus}.

    \begin{lemma}[{\cite[Lemma~8]{bodlaender_partial_1998}}] \label{lem:bodlaender8}  \label{lem:bodlaender8pw}
        Let $G$ be a graph and $k \geq 0$ such that $\tw G \leq k$ and $|V(G)| \geq k+1$.
        Then $G$ has a tree decomposition $\beta \colon V(T) \to 2^{V(G)}$ such that
        $\lvert \beta(t) \rvert = k+1$ for all $t \in V(T)$ and
        $\lvert \beta(s) \cap \beta(t) \rvert = k$ for all $st \in E(T)$.
        Furthermore, if $\pw G \leq k$, then $T$ can be chosen to be a path.
    \end{lemma}
    \begin{proof}
        Suppose that the treewidth of $G$ is $\ell \leq k$. 
        If $|V(G)| = k+1$ then the tree decomposition over the one-vertex tree is as desired.
        Otherwise, any tree decomposition of width at most $k$ must be over a graph on at least two vertices.
        Let $\beta \colon V(T) \to 2^{V(G)}$ be any tree decomposition of width $\ell$. 
        We repeatedly apply the following steps:
        \begin{itemize}
            \item If $st \in E(T)$ is such that $\beta(s) \subseteq \beta(t)$ or $\beta(t) \subseteq \beta(s)$ then the edge $st$ in $T$ can be contracted and the set $\beta(s) \cup \beta(t)$  can be taken to be the bag at the vertex obtained by contraction.
            \item If $st \in E(T)$ and $|\beta(s)| < k+1$ and $\beta(t) \not\subseteq \beta(s)$ then $\beta(s)$ can be enlarged by a vertex $v \in \beta(s) \setminus \beta(t)$.
            \item If $st \in E(T)$ and $|\beta(s)| = |\beta(t)| = k+1$ and $|\beta(s) \cap \beta(t)| < k$ then subdivide the edge $st$ in $T$ by introducing a fresh vertex $r$.
            Choose vertices $v \in \beta(s) \setminus \beta(t)$ and $w \in \beta(t) \setminus \beta(s)$ and let $\beta(r) \coloneqq (\beta(s) \setminus \{v\}) \cup \{w\}$.
        \end{itemize}
        If none of these operations can be applied, the tree decomposition is as desired.

        The operations used  to manipulate the decomposition tree were contraction and subdivision. If the initial decomposition tree is in fact a path then the resulting tree will also be a path. This yields the last assertion.
    \end{proof}

	\subsection{Systems of Equations for Graph Isomorphism}
    \label{sec:systems}

    Let $G$ and $H$ be simple graphs with vertices $g_1, \dots, g_\ell \in V(G)$ and $h_1, \dots, h_\ell \in V(H)$ for $\ell \geq 1$.
    The \emph{atomic type} of a tuple of vertices of a graph is defined as follows:
    Let $\rel_G(g_1, \dots, g_\ell) = \rel_H(h_1,\dots, h_\ell)$ if $g_i = g_j \Leftrightarrow h_i = h_j$ and $g_ig_j \in E(G) \Leftrightarrow h_ih_j \in E(H)$ for all $i, j \in [\ell]$. 
    In this case, the set $\{g_1h_1, \dots g_\ell h_\ell\} \in \binom{V(G) \times V(H)}{\ell}$ is  called a \emph{partial isomorphism}.
    
    Two simple graphs $G$ and $H$ are isomorphic if and only if there exists a $\{0,1\}$-solution to quadratic integer program $\ISO(G, H)$ which comprises variables $X_{gh}$ for $gh \in V(G) \times V(H)$ and equations
	\begin{align}
		\sum_{h \in V(H)} X_{gh} -1 &= 0 && \label{s1} \text{for all } g \in V(G), \\
		\sum_{g \in V(G)} X_{gh} - 1 &= 0 &&\label{s2} \text{for all } h \in V(H), \\
		X_{gh} X_{g'h'} &= 0 && \text{for all $gh, g'h' \in V(G) \times V(H)$ s.t.\  $\rel_G(g,g') \neq \rel_H(h, h')$.} \label{type}
	\end{align}
 
    We define the Lasserre relaxation of $\ISO(G, H)$ following \cite{mancinska_quantum_2019}.  
    See also \cref{app:lasserre} for a comparison to the version used in \cite{atserias_definable_2023}.
	
	\begin{definition} \label{def:lasserre}
		Let $t \geq 1$. The \emph{level-$t$ Lasserre relaxation for graph isomorphism} has variables $y_I$ ranging over $\mathbb{R}$ for $I \in \binom{V(G) \times V(H)}{\leq 2t}$. The constraints are
		\begin{align}
			M_t(y) \coloneqq (y_{I \cup J})_{I, J \in \binom{V(G) \times V(H)}{\leq t}}  &\succeq 0,&&\label{lassere1} \\
			\sum_{h \in V(H)} y_{I \cup \{gh\}} &= y_{I} &&
            \parbox{9cm}{for all $I \in \binom{V(G) \times V(H)}{\leq 2t-2}$ and all $g \in V(G)$,}
            \label{lassere2} \\
			\sum_{g \in V(G)} y_{I \cup \{gh\}} &= y_{I} && \parbox{9cm}{for all $I \in \binom{V(G) \times V(H)}{\leq 2t-2}$ and all $h \in V(H)$,} \label{lassere3} \\
			y_I & = 0 && \parbox{9cm}{for all $I \in \binom{V(G) \times V(H)}{\leq 2t}$  such that $I$ is not a partial isomorphism,} \label{lassere4}\\
			y_\emptyset &= 1.&& \label{lassere5}
        \end{align}
        If the system is feasible for two graphs $G$ and $H$, write $G \simeq_t^{\textup{L}} H$.
        If the system together with the constraint $y_I \geq 0$ for all $I \in \binom{V(G)  \times V(H}{\leq 2t}$ is feasible, write $G \simeq_t^{\textup{L}^+} H$.
	\end{definition}
    In \cref{thm:lasserre-s+-iso}, we show that the $2t-2$ in \cref{lassere2,lassere3} can be replaced by $2t-1$ without loss of generality.
    That is, the system in \cref{def:lasserre} has a (non-negative)  real solution if and only if the system obtained replacing \cref{lassere2,lassere3} with \cref{lassere2a,lassere3a} has a (non-negative) real solution.
    The $2t-2$ in \cref{def:lasserre} is an artefact of its construction from $\ISO(G, H)$, cf.\ \cite[Section~10]{mancinska_graph_2020} and \cite[Equations~(2d)--(2e)]{snook_graph_2014}.
    
    The second hierarchy of integer programming relaxation considered in this article is the Sherali--Adams relaxation \cite{sherali_hierarchy_1990}.
    It has been applied both to the  integer linear program and to $\ISO(G, H)$, the integer quadratic program for graph isomorphism. 
    For the linear program, it was shown in \cite{atserias_sheraliadams_2013} that the Sherali--Adams levels are sandwiched between the levels of the Weisfeiler--Leman hierarchy. 
    Subsequently,  variants of these linear programs were proposed in \cite{grohe_pebble_2015} which correspond precisely to the levels of the latter hierarchy. 
    In this work, we focus on the Sherali--Adams relaxations of the integer quadratic program $\ISO(G, H)$.
    See \cite[Section~2.7]{grohe_homomorphism_2021_arxiv} for a definition.
    In \cite{malkin_sheraliadams_2014}, it was shown that the level-$t$ Sherali--Adams relaxation of $\ISO(G,H)$ has a non-negative rational solution if and only if $G$ and $H$ are not distinguished by the $(t-1)$-dimensional Weisfeiler--Leman algorithm. 
    The following \cref{thm:sa} summarises equivalent formulations.
    \begin{theorem}[\cite{malkin_sheraliadams_2014,dvorak_recognizing_2010,cai_optimal_1992}] \label{thm:sa}
        Let $t \geq 1$. For graphs $G$ and $H$, the following are equivalent:
        \begin{enumerate}
            \item the level-$t$ Sherali--Adams relaxation of $\ISO(G,H)$ has a non-negative rational solution, i.e.\  $G \simeq_t^{\textup{SA}} H$,
            \item $G$ and $H$ satisfy the same sentences of $t$-variable first order logic with counting quantifiers,
            \item $G$ and $H$ are homomorphism indistinguishable over the graphs of treewidth at most $t-1$,
            \item $G$ and $H$ are not distinguished by the $(t-1)$-dimensional Weisfeiler--Leman algorithm,
        \end{enumerate}
    \end{theorem}

	\section{From Lasserre to Homomorphism Tensors}
    \label{sec:lasserre-to-tensor}

    In this section, the tools are developed which will be used to translate a solution to the level-$t$ Lasserre relaxation into a statement on homomorphism indistinguishability.
	For this purpose, three equivalent characterisations of $\simeq_t^{\text{L}}$ and $\simeq_t^{\text{L}^+}$ are introduced.
	    \Cref{thm:summary,thm:summary-pos} summarise our results. The notions in items 2--4 and the graph classes $\mathcal{L}_t$ and $\mathcal{L}_t^+$ are defined in \cref{sec:K-isomorphic,sec:choi,sec:matrix-algebras,sec:hi}, respectively.
    Most of the proofs are of a linear algebraic nature. Graph theoretical repercussions are discussed in \cref{sec:hi}.

    \begin{theorem} \label{thm:summary}
		Let $t \geq 1$.
		For graphs $G$ and $H$, the following are equivalent:
		\begin{enumerate}
			\item the level-$t$ Lasserre relaxation of $\ISO(G, H)$ is feasible,
    		\item $G$ and $H$ are level-$t$ $\mathcal{PSD}$-isomorphic, cf.\ \cref{def:K-isomorphic},
    		\item there is a level-$t$ $\mathcal{PSD}$-isomorphism map from $G$ to $H$, cf.\ \cref{thm4.7},
    		\item $G$ and $H$ are partially $t$-equivalent, cf.\ \cref{def:pce},
    		\item $G$ and $H$ are homomorphism indistinguishable over $\mathcal{L}_t$, cf.\ \cref{def:l-lplus}.
		\end{enumerate}
	\end{theorem}
	
	\begin{theorem} \label{thm:summary-pos}
		Let $t \geq 1$.
		For graphs $G$ and $H$, the following are equivalent:
		\begin{enumerate}
			\item the level-$t$ Lasserre relaxation of $\ISO(G, H)$ with non-negativity constraints is feasible,
    		\item $G$ and $H$ are level-$t$ $\mathcal{DNN}$-isomorphic, cf.\ \cref{def:K-isomorphic},
    		\item there is a level-$t$ $\mathcal{DNN}$-isomorphism map from $G$ to $H$, cf.\ \cref{thm4.7},
    		\item $G$ and $H$ are $t$-equivalent, cf.\ \cref{def:ce},
    		\item $G$ and $H$ are homomorphism indistinguishable over $\mathcal{L}_t^+$, cf.\ \cref{def:l-lplus}.
		\end{enumerate}
	\end{theorem}

    Variants of the notions in items 2--4 have already been defined  for the case $t=1$ in  \cite{mancinska_graph_2020}.
    Our contribution amounts to extending these definitions to the entire Lasserre hierarchy. 
    A recurring theme in this context is accounting for additional symmetries. 
    The variables $y_I$ of the Lasserre system of equations, cf.\  \cref{def:lasserre}, are indexed by sets of vertex pairs rather than by tuples of such. 
    Hence, when passing from such variables to tuple-indexed matrices, one must impose the additional symmetries arising this way.
    This is formalised at various points using an action of the symmetric group on the axes of the matrices.
    In the case $t = 1$, 
    such a set-up is not necessary since indices $I$ are of size at most $2$ and all occurring matrices can be taken to be invariant under transposition.

    In the subsequent sections, \cref{thm:summary,thm:summary-pos} will be proven in parallel. The equivalence of items 1 and 2, 2 and 3, and 3 and 4 are established in \cref{sec:connection-lasserre}, \cref{sec:choi}, and \cref{sec:matrix-algebras}, respectively. The statements on homomorphism indistinguishability are proven in \cref{sec:hi}.

    \subsection{Isomorphism Relaxations via Matrix Families}
	\label{sec:K-isomorphic}

    In this section,
    as a first step towards proving \cref{thm:summary,thm:summary-pos}, the notion of level-$t$ $\mathcal{K}$-isomorphic graphs for arbitrary families of matrices $\mathcal{K}$ is introduced. 
    In \cite{mancinska_graph_2020}, level-$1$ $\mathcal{K}$-isomorphic graphs where studied for various families of matrices $\mathcal{K}$.
    In this work, the main interest lies on the family of positive semidefinite matrices $\mathcal{PSD}$ and the family of entry-wise non-negative positive semidefinite matrices $\mathcal{DNN}$.
    Level-$t$ isomorphism for these families is proven to correspond to  $\simeq_t^{\text{L}}$ and $\simeq_t^{\text{L}^+}$ respectively, cf.\  \cref{thm:lasserre-s+-iso,thm:lasserre-dnn-iso}.

	\begin{definition} \label{def:K-isomorphic}
		Let $\mathcal{K}$ be a family of matrices. 
        Graphs $G$ and $H$ are said to be \emph{level-$t$ $\mathcal{K}$-isomorphic}, in symbols $G \cong^t_\mathcal{K} H$, if there is a matrix $M \in \mathcal{K}$ with rows and columns indexed by $(V(G) \times V(H))^t$ such that for every $g_1h_1 \dots g_th_t, g_{t+1}h_{t+1} \dots g_{2t}h_{2t} \in (V(G) \times V(H))^t$ the following equations hold:
		\begin{align}
            \intertext{For every $i \in [2t]$,}
			\sum_{g_i \in V(G)} M_{g_1h_1 \dots g_th_t, g_{t+1}h_{t+1} \dots g_{2t}h_{2t}} &= \sum_{h_i\in V(H)} M_{g_1h_1 \dots g_th_t, g_{t+1}h_{t+1} \dots g_{2t}h_{2t}}, \label{m1}\\
			\sum_{g_1, \dots, g_{2t} \in V(G)} M_{g_1h_1 \dots g_th_t, g_{t+1}h_{t+1} \dots g_{2t}h_{2t}}&= 1 = 	\sum_{h_1, \dots, h_{2t} \in V(H)} M_{g_1h_1 \dots g_th_t, g_{t+1}h_{t+1} \dots g_{2t}h_{2t}}. \label{m2}
            \intertext{If $\rel_G(g_1, \dots, g_{2t}) \neq  \rel_H(h_1, \dots, h_{2t})$ then }
			M_{g_1h_1 \dots g_th_t, g_{t+1}h_{t+1} \dots g_{2t}h_{2t}} &= 0. \label{m3}
            \intertext{For all  $\sigma \in \mathfrak{S}_{2t}$,}
			M_{g_1h_1 \dots g_th_t, g_{t+1}h_{t+1} \dots g_{2t}h_{2t}} &= M_{g_{\sigma(1)}h_{\sigma(1)} \dots g_{\sigma(t)}h_{\sigma(t)}, g_{\sigma(t+1)}h_{\sigma(t+1)} \dots g_{\sigma(2t)}h_{\sigma(2t)}} \label{m4}.
		\end{align}
	\end{definition}
	
	Note that for  $t = 1$ and any family of matrices $\mathcal{K}$ closed under taking transposes \cref{m4} is vacuous.

    Systems of equations comparing graphs akin to \cref{m1,m2,m3,m4} were also studied by~\cite{grohe_homomorphism_2022}.
    Feasibility of such equations is typically invariant under taking the complements of the graphs as remarked below. 
    This semantic property of the relation $\cong^t_\mathcal{K}$ is relevant in the context of homomorphism indistinguishability as shown by \cite{seppelt_logical_2023}.

	\begin{remark} \label{rem:complements}
		For a simple graph $G$, write $\overline{G}$ for its complement, i.e.\  $V(\overline{G}) \coloneqq V(G)$ and $E(\overline{G}) \coloneqq \binom{V(G)}{2} \setminus E(G)$. For all graphs $G$ and $H$ and 
		$g_1, \dots, g_{2t} \in V(G)$, $h_1, \dots, h_{2t} \in V(H)$, it holds that
		\[
			\rel_G(g_1, \dots, g_{2t}) = \rel_H(h_1, \dots, h_{2t})
			\iff
			\rel_{\overline{G}}(g_1, \dots, g_{2t}) = \rel_{\overline{H}}(h_1, \dots, h_{2t}).
		\]
		Thus, $G \cong^t_\mathcal{K} H$ if and only if $\overline{G} \cong^t_\mathcal{K} \overline{H}$ for all families of matrices $\mathcal{K}$ and $t \in \mathbb{N}$.
	\end{remark}

	\subsection{Choi Matrices and Isomorphism Maps}
    \label{sec:choi}
    
    In this section, an alternative characterisation for level-$t$ $\mathcal{K}$-isomorphism is given.
    Intuitively, the indices of the matrix $M \in \mathbb{C}^{(V(G) \times V(H))^t \times (V(G) \times V(H))^t}$ 
    from \cref{def:K-isomorphic} are regrouped yielding a linear map $\Phi \colon \mathbb{C}^{V(G)^t \times V(G)^t} \to \mathbb{C}^{V(H)^t \times V(H)^t}$. In linear algebraic terms, $M$ is the Choi matrix of $\Phi$.
    The map $\Phi$ will later be interpreted as a function sending homomorphism tensors of $(t,t)$-bilabelled graphs $\boldsymbol{F}_G \in \mathbb{C}^{V(G)^t \times V(G)^t}$ with respect to $G$ to their counterparts $\boldsymbol{F}_H$ for $H$.
    
    The most basic bilabelled graphs, so-called \emph{atomic} graphs, make their first appearance in \cref{thm4.7}.
    These graphs are used to reformulate \cref{m3,lassere4}. The atomic graphs are also the graphs which the sets $\mathcal{L}_t$ and $\mathcal{L}_t^+$ of \cref{thm:main2,thm:main3} are generated by, cf.\  \cref{def:l-lplus}. Examples are depicted in \cref{fig:atomic,fig:atomic-t1}.

    \begin{figure}
        \centering
        \begin{subfigure}[t]{.3\linewidth}
            \centering
            \includegraphics[page=1,scale=1.2]{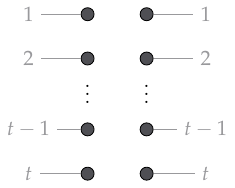}
            \caption{$\boldsymbol{J}$}
        \end{subfigure}
        \begin{subfigure}[t]{.3\linewidth}
            \centering
            \includegraphics[page=2,scale=1.2]{figure-basal.pdf}
            \caption{$\boldsymbol{A}^{2,2t}$}
        \end{subfigure}
        \begin{subfigure}[t]{.3\linewidth}
            \centering
            \includegraphics[page=3,scale=1.2]{figure-basal.pdf}
            \caption{$\boldsymbol{I}^{2,2t}$}
        \end{subfigure}
        \caption{Examples of the atomic graphs from \cref{def:atomic}. The gray lines (the \emph{wires} \cite{mancinska_quantum_2019}) indicate the in-labels (left) and out-labels (right). }
        \label{fig:atomic}
    \end{figure}
 
	\begin{definition}  \label{def:atomic}\label{obs:atomic}
		Let $t \geq 1$.
		A $(t,t)$-bilabelled graph $\boldsymbol{F} = (F, \boldsymbol{u}, \boldsymbol{v})$ is \emph{atomic} if all its vertices are labelled. Write $\mathcal{A}_t$ for the set of $(t,t)$-bilabelled atomic graphs. Note that the  set of atomic graphs $\mathcal{A}_t$ is generated under parallel composition by the graphs
		\begin{itemize}
			\item $\boldsymbol{J} \coloneqq (J, (1,\dots, t), (t+1, \dots, 2t))$ with $V(J) = [2t]$, $E(J) = \emptyset$,
			\item $\boldsymbol{A}^{ij} \coloneqq (A^{ij}, (1,\dots, t), (t+1, \dots, 2t))$ with $V(A^{ij}) = [2t]$, $E(A^{ij}) = \{ij\}$ for $1 \leq i < j \leq 2t$,
			\item $\boldsymbol{I}^{ij}$ for $1 \leq i < j \leq 2t$ which is obtained from $\boldsymbol{A}^{ij}$ by contracting the edge $ij$ and removing the resulting loop.
		\end{itemize}
	\end{definition}

    The following \cref{thm4.7} relates the properties of $\Phi$ and $M$. 
    In \cref{i3}, $J$ denotes the all-ones matrix of appropriate dimension.

	\begin{theorem} \label{thm4.7}
		Let $t \geq 1$. Let $G$ and $H$ be graphs and $\mathcal{K} \in \{\mathcal{DNN}, \mathcal{PSD}\}$ be a family of matrices.
		Let $\Phi \colon \mathbb{C}^{V(G)^t \times V(G)^t} \to \mathbb{C}^{V(H)^t \times V(H)^t}$ be a linear map. Then the following are equivalent.
		\begin{enumerate}
			\item The Choi matrix $C_\Phi$ of $\Phi$ satisfies \cref{m1,m2,m3,m4} and $C_\Phi \in \mathcal{K}$,\label{item:thm4.7.1}
			\item $\Phi$ is a \emph{level-$t$ $\mathcal{K}$-isomorphism map from $G$ to $H$}, i.e.\  it satisfies\label{item:thm4.7.2}
			\begin{align}
				&\Phi \text{ is completely } \mathcal{K}\text{-preserving}, \label{i1}\\
				&\Phi(\boldsymbol{A}_G \odot X) = \boldsymbol{A}_H \odot \Phi(X) \text{ for all atomic } \boldsymbol{A} \in \mathcal{A}_t \text{ and all } X \in \mathbb{C}^{V(G)^t \times V(G)^t},\label{i2}\\
				&\Phi(J) = J = \Phi^*(J),\label{i3}\\
				&\Phi(X^\sigma) = \Phi(X)^\sigma \text{ for all } \sigma \in \mathfrak{S}_{2t} \text{ and all } X \in \mathbb{C}^{V(G)^t \times V(G)^t}.\label{i4}
			\end{align}
			\item $\Phi^*$ is a level-$t$ $\mathcal{K}$-isomorphism map from $H$ to $G$.\label{item:thm4.7.3}
		\end{enumerate}
	\end{theorem}

    We remark that \cref{thm4.7} and in particular its \cref{i3} have brought us closer to interpreting the Lasserre system of equation from the perspective of homomorphism indistinguishability.
    As argued in \cref{rem:num-vertices}, the map $\Phi$, which will be understood as mapping homomorphism tensors $\boldsymbol{F}_G$ to $\boldsymbol{F}_H$, is sum-preserving. 
    Since the sum of the entries of these tensors equals the number of homomorphisms from their underlying unlabelled graphs to $G$ and $H$, respectively,
    this is relevant
    for establishing a connection between $\mathcal{K}$-isomorphism maps  and homomorphism indistinguishability.
    
	\begin{remark} \label{rem:num-vertices}
		If a linear map $\Phi \colon \mathbb{C}^{n \times n} \to \mathbb{C}^{m \times m}$ is such that $J = \Phi^*(J)$ then it is \emph{sum-preserving}, i.e.\  $\soe(X) = \soe(\Phi(X))$ for all $X \in \mathbb{C}^{n \times n}$.
		Indeed, $\soe(X) = \left< X, J \right> = \left< X, \Phi^*(J) \right> = \left< \Phi(X), J \right> = \soe( \Phi(X))$ where $\left< A, B \right> \coloneqq \tr(AB^*)$.
		In particular, if there is $\Phi$ satisfying \cref{i2,i3} for graphs $G$ and $H$ then $|V(G)| = |V(H)|$.
	\end{remark}
	
    Equipped with \cref{rem:num-vertices}, we conduct the proof of \cref{thm4.7}.
    
    \begin{proof}[Proof of Theorem~\ref{thm4.7}]
        The equivalence of \cref{item:thm4.7.2,item:thm4.7.3} follows immediately from \cref{lem:mg-4.4,lem:mg-4.5}.

        For the equivalence of \cref{item:thm4.7.1,item:thm4.7.2}, first note that, by \cref{lem:mg-4.4}, $C_\Phi \in \mathcal{K}$ if and only if Property~\eqref{i1} holds.
        Moreover,
        for $g_1, \dots, g_{2t} \in V(G)$ and $h_1, \dots, h_{2t} \in V(H)$, the assertions
        \[
        \forall \boldsymbol{A} \in \mathcal{A}(t,t),\quad \boldsymbol{A}_G(g_1 \dots g_t, g_{t+1} \dots g_{2t})
        =
        \boldsymbol{A}_H(h_1 \dots h_t, h_{t+1} \dots h_{2t})
        \]
        and $\rel_G(g_1, \dots, g_{2t}) = \rel_H(h_1, \dots, h_{2t})$ are equivalent. By \cref{lem:mg-4.5},
        \cref{i2,m3} are equivalent. 
        Furthermore, \cref{i4,m4}  and \cref{i3,m2} are respectively equivalent.
        
        Finally, we argue that \cref{item:thm4.7.2,item:thm4.7.3} imply \cref{m1}.
        To that end, consider the atomic graph $\boldsymbol{K}^j \in \mathcal{A}(t,t)$ for $j \in [t]$ as defined in \cref{eq:kj} and depicted by \cref{fig:kj}.
        \begin{figure}
            \centering
            \includegraphics[page=4,scale=1.2]{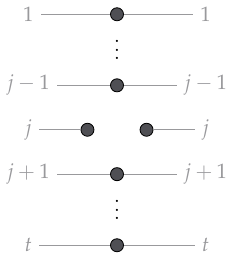}
            \caption{The atomic graph $\boldsymbol{K}^j$ as defined in \cref{eq:kj}.}
            \label{fig:kj}
        \end{figure}
        \begin{equation}\label{eq:kj}
            \boldsymbol{K}^j \coloneqq \boldsymbol{I}^{1, t+1} \odot \dots \odot \boldsymbol{I}^{j-1, t+j-1} \odot \boldsymbol{I}^{j+1, t+j+1} \odot \dots \odot \boldsymbol{I}^{t, 2t}. 
        \end{equation}
        In order to apply \cref{lem:mg-4.9}, we first argue that $\Phi$ is trace-preserving. 
        Let 
        \[
            \boldsymbol{I} \coloneqq \boldsymbol{I}^{1, t+1} \odot \dots \odot \boldsymbol{I}^{t, 2t} \in \mathcal{A}(t,t).
        \]
        By \cref{i3,rem:num-vertices}, $\Phi$ is sum-preserving.
        For every $X \in \mathbb{C}^{V(G)^t \times V(G)^t}$, 
        \[
            \tr(\Phi(X))
            = \soe(\boldsymbol{I}_H \odot \Phi(X))
            \stackrel{\eqref{i2}}{=} \soe( \Phi(\boldsymbol{I}_G \odot X))
            = \soe(\boldsymbol{I}_G \odot X)
            = \tr(X).
        \]
        Thus, \cref{lem:mg-4.9,i2} yield that for all $j \in [t]$ and all $X \in \mathbb{C}^{V(G)^t \times V(G)^t}$, 
        \begin{equation} \label{eq:helper-kj}
            \Phi( \boldsymbol{K}^j_G X) = \Phi(\boldsymbol{K}^j_G) \Phi(X)
            \quad \text{and} \quad
            \Phi(  X\boldsymbol{K}^j_G) = \Phi(X)\Phi(\boldsymbol{K}^j_G).
        \end{equation}
        Next, we substitute standard basis elements for $X$ in \cref{eq:helper-kj}.
        For $g_1,\dots, g_{2t} \in V(G)$, write $E^{g_1\dots g_{2t}} \in \mathbb{C}^{V(G)^t \times V(G)^t}$ for the corresponding standard basis vector.
        To ease notation, we verify \cref{m1} for $i=1$.
        For all $g_1,\dots, g_{2t} \in V(G)$ and $h_1, \dots, h_{2t} \in V(H)$,
        \begin{align*}
            \sum_{g \in V(G)} M_{gh_1 g_2h_2 \dots g_{2t} h_{2t}}
            &= 	\sum_{g \in V(G)} \Phi_{h_1\dots h_{2t}, gg_2 \dots g_{2t}} \\
            &= \sum_{g\in V(G)} \Phi(E^{gg_2\dots g_{2t}})_{h_1\dots h_{2t}} \\
            &= \Phi(\boldsymbol{K}^1_G E^{g_1\dots g_{2t}})_{h_1\dots h_{2t}} \\
            & \stackrel{\mathclap{\eqref{eq:helper-kj}}}{=} \  (\Phi(\boldsymbol{K}^1_G)  \Phi(E^{g_1\dots g_{2t}}))_{h_1\dots h_{2t}} \\
            & \stackrel{\mathclap{\eqref{i2}}}{=} \  ( \boldsymbol{K}^1_H  \Phi(E^{g_1\dots g_{2t}}))_{h_1\dots h_{2t}} \\
            &= \sum_{h \in V(H)} \Phi_{hh_2 \dots h_{2t}, g_1 \dots g_{2t}} \\
            &= \sum_{h \in V(G)} M_{g_1h g_2h_2 \dots g_{2t} h_{2t}},
        \end{align*}
        as desired.
	\end{proof}

	\subsection{From \texorpdfstring{$\mathcal{K}$}{K}-Isomorphism Maps to the Lasserre Hierarchy }
    \label{sec:connection-lasserre}

    By the following \cref{thm:lasserre-s+-iso,thm:lasserre-dnn-iso}, the notions introduced in \cref{def:K-isomorphic,thm4.7} are equivalent to the object of our main interest, namely feasibility of the level-$t$ Lasserre relaxation with and without non-negativity constraints. Our results extend those of \cite[Lemma~10.1]{mancinska_graph_2020} to the entire Lasserre hierarchy.

	\begin{theorem} \label{thm:lasserre-s+-iso}
		Let $t \geq 1$.
		Two graphs $G$ and $H$ are level-$t$ $\mathcal{PSD}$-isomorphic if and only if  $G \simeq_t^{\textup{L}} H $. \end{theorem}

    \begin{proof}Suppose that $(y_I)_{I \in \binom{V(G) \times V(H)}{ \leq 2t}}$ is a solution to \cref{lassere1,lassere2,lassere3,lassere4,lassere5}. 
		It is argued that the matrix defined via $M_{g_1h_1 \dots g_th_t, g_{t+1}h_{t+1} \dots g_{2t}h_{2t}} \coloneqq y_{\{g_1h_1, \dots, g_{2t}h_{2t}\}}$ satisfies \cref{m1,m2,m3,m4}. \Cref{m3} follows directly from \cref{lassere4}. \Cref{m4} is immediate from the definition.
		
		By \cref{lassere1}, let $v_I$ for $I \in \binom{V(G) \times V(H)}{ \leq t}$ be vectors such that $y_{I \cup J} = \left< v_I, v_J \right>$ for $I, J \in \binom{V(G) \times V(H)}{ \leq t}$. Then
		\[
			M_{g_1h_1 \dots g_th_t, g_{t+1}h_{t+1}} = y_{\{g_1h_1, \dots, g_{2t}h_{2t}\}} = \left< v_{\{g_1h_1, \dots, g_{t}h_{t}\}}, v_{\{g_{t+1}h_{t+1}, \dots, g_{2t}h_{2t}\}} \right>.
		\]
		Thus, $M$ is positive semidefinite.
		It remains to verify \cref{m1,m2}.
		
		\begin{claim} \label{cl1}
			Every $I \in \binom{V(G) \times V(H)}{\leq t-1}$ satisfies $\sum_{g \in V(G)} v_{I \cup \{gh\}} = v_I = \sum_{h \in V(H)} v_{I \cup \{gh\}}$.
		\end{claim}
		\begin{subproof}
  		    Recall that $y_{I \cup J} = \left< v_I, v_J \right>$ for $I, J \in \binom{V(G) \times V(H)}{ \leq t}$.
			By \cref{lassere4,lassere3},
			\begin{align*}
				\left\langle \sum_{g \in V(G)} v_{I \cup \{gh\}}, \sum_{g \in V(G)} v_{I \cup \{gh\}} \right\rangle
				&= \sum_{g, g' \in V(G)} \left\langle v_{I \cup \{gh\}}, v_{I \cup \{g'h\}} \right\rangle \\
				&= \sum_{g, g' \in V(G)} y_{I \cup \{gh\} \cup \{g'h\}} \\
				&\stackrel{\eqref{lassere4}}{=} \sum_{g \in V(G)} y_{I \cup \{gh\}}
				\stackrel{\eqref{lassere3}}{=} y_I.
			\end{align*}
			Observe that \cref{lassere3} indeed applies since $t - 1 \leq 2t -2$ for all $t \geq 1$. Moreover,
			\begin{align*}
				\left< v_I, \sum_{g \in V(G)} v_{I \cup \{gh\}} \right>
				= \sum_{g \in V(G)} y_{I \cup \{gh\}} = y_I
			\end{align*}
			and hence combining the above equalities,
			\begin{align*}
				\left\lVert v_I - \sum_{g \in V(G)} v_{I \cup \{gh\}} \right\rVert^2 = y_I - 2y_I + y_I = 0
			\end{align*}
			The claim is proven analogously when summation is over $h \in V(H)$.
		\end{subproof}
		\cref{cl1,lassere5} imply \cref{m2}. Indeed,
		\begin{align*}
			\sum_{g_1 \dots g_{2t} \in V(G)} M_{g_1h_1 \dots g_th_t, g_{t+1}h_{t+1} \dots g_{2t}h_{2t}}
			&= \sum_{g_1 \dots g_{2t} \in V(G)} y_{\{g_1h_1, \dots, g_th_t\} \cup \{ g_{t+1}h_{t+1} \dots g_{2t}h_{2t}\}} \\
			&= \sum_{g_1 \dots g_{2t} \in V(G)} \left< v_{\{g_1h_1, \dots, g_th_t\}}, v_{\{ g_{t+1}h_{t+1} \dots g_{2t}h_{2t}\}} \right> \\
			&= \left< v_\emptyset, v_\emptyset \right> \\
			&= y_\emptyset\\
			&= 1.
		\end{align*}
		Moreover, for \cref{m1}, letting $i = 1$ to ease notation,
		\begin{align*}
			\sum_{g_1 \in V(G)} M_{g_1h_1 \dots g_th_t, g_{t+1}h_{t+1} \dots g_{2t}h_{2t}}
			&= \sum_{g_1 \in V(G)} y_{\{g_1h_1\} \cup \{g_2h_2 \dots g_th_t\} \cup \{g_{t+1}h_{t+1} \dots g_{2t}h_{2t}\}} \\
			&= \sum_{g_1 \in V(G)} \left<  v_{\{g_1h_1\} \cup \{g_2h_2 \dots g_th_t\} }, v_{\{g_{t+1}h_{t+1} \dots g_{2t}h_{2t}\}} \right> \\
			&= \sum_{h_1 \in V(G)} \left<  v_{\{g_1h_1\} \cup \{g_2h_2 \dots g_th_t\} }, v_{\{g_{t+1}h_{t+1} \dots g_{2t}h_{2t}\}} \right> \\
			&= \sum_{h_1 \in V(G)} M_{g_1h_1 \dots g_th_t, g_{t+1}h_{t+1} \dots g_{2t}h_{2t}}.
		\end{align*}
		This concludes the proof that $M$ satisfies \cref{m1,m2,m3,m4}.
	
		Conversely, let $v_{\vec{I}}$ for $\vec{I} \in (V(G) \times V(H))^t$ denote the Gram vectors of a matrix $M$ satisfying \cref{m1,m2,m3,m4}.
		Define $v_I \coloneqq v_{\vec{I}}$ for $|I| = t$ and any ordering. By \cref{m4}, $v_I$ is well-defined. Let furthermore,
		\[
			v_I^{g_{i+1}\dots g_{t}} \coloneqq \sum_{h_{i+1}, \dots, h_t \in V(H)} v_{\vec{I}g_{i+1}h_{i+1}\dots g_th_t}
		\]
		for $I \in \binom{V(G) \times V(H)}{i}$ and $g_{i+1} \dots g_t \in V(G)^{t-i}$.
		Define $v_I^{h_{i+1}\dots h_{t}}$ analogously.
		
		\begin{claim} \label{cl2}
			For all $g_{i+1} \dots g_t, g'_{i+1} \dots g'_t \in V(G)^{t-i}$,
            it holds that
			$v_I^{g_{i+1}\dots g_{t}} = v_I^{g'_{i+1}\dots g'_{t}}$.
		\end{claim}
		\begin{subproof}
			By definition, the term $\left\lVert v_I^{g_{i+1}\dots g_{t}} - v_I^{g'_{i+1}\dots g'_{t}} \right\rVert^2$ is equal to
			\begin{align*}
				\sum_{\substack{h_{i+1}, \dots, h_t \in V(H), \\ h'_{i+1}, \dots, h'_t \in V(H)}} \left(  M_{\vec{I}g_{i+1}h_{i+1}\dots g_th_t, \vec{I}g_{i+1}h'_{i+1}\dots g_th'_t} 
                - 2 M_{\vec{I}g_{i+1}h_{i+1}\dots g_th_t, \vec{I}g'_{i+1}h'_{i+1}\dots g'_th'_t} 
                + M_{\vec{I}g'_{i+1}h_{i+1}\dots g'_th_t, \vec{I}g'_{i+1}h'_{i+1}\dots g'_th'_t} \right).
			\end{align*}
			By \cref{m1}, this expression is zero.
		\end{subproof}
	
		By \cref{cl2}, the reference to $g_{i+1}\dots g_{t}$ can be dropped, yielding vectors $v_I^G$ and $v_I^H$. It follows that
		\begin{align*}
			|V(G)|^{t-i} v_I^G 
			&= \sum_{g_{i+1}\dots g_{t} \in V(G)^{t-i}} v_I^{g_{i+1}\dots g_{t}}
			= \sum_{\substack{g_{i+1}\dots g_{t} \in V(G)^{t-i}\\ h_{i+1}\dots h_{t} \in V(H)^{t-i}}} v_{\vec{I}g_{i+1}h_{i+1}\dots g_{t}h_t} \\
			&= \sum_{h_{i+1}\dots h_{t} \in V(H)^{t-i}} v_I^{h_{i+1}\dots h_{t}}
			= |V(H)|^{t-i} v_I^H.
		\end{align*}
		This implies that $v_I^G = v_I^H$ since $G$ and $H$ have the same number of vertices, cf.\  \cref{rem:num-vertices}. Let $v_I \coloneqq v_I^G = v_I^H$.
        The following \cref{cl:set-ip} is immediate from \cref{m4}:
  
		\begin{claim} \label{cl:set-ip}
			If $I \cup J = I' \cup J'$ for $I, I', J, J' \in \binom{V(G) \times V(H)}{\leq t}$ then $\left< v_I, v_J \right> = \left< v_{I'}, v_{J'} \right>$.
		\end{claim}
	
		Hence, $y_I$ for $I \in \binom{V(G) \times V(H)}{\leq 2t}$ can be set to $\left< v_{I'}, v_{I''} \right>$ for any $I', I'' \in \binom{V(G) \times V(H)}{\leq t}$ such that $I = I' \cup I''$.
		Then \cref{lassere1,lassere2,lassere3} holds by construction.
      In fact, it follows that \cref{lassere2a,lassere3a} below, which imply \cref{lassere2,lassere3}, hold:
        \begin{align}
            \sum_{h \in V(H)} y_{I \cup \{gh\}} &= y_{I} &&\parbox{9cm}{for all $I \in \binom{V(G) \times V(H)}{\leq 2t-1}$ and all $g \in V(G)$,} \label{lassere2a} \\
            \sum_{g \in V(G)} y_{I \cup \{gh\}} &= y_{I} && \parbox{9cm}{for all $I \in \binom{V(G) \times V(H)}{\leq 2t-1}$ and all $h \in V(H)$.} \label{lassere3a} 	
        \end{align}
		\cref{lassere4} follows from \cref{m3}.
	\end{proof}

    The following \cref{thm:lasserre-dnn-iso} is proven analogously, observing that the construction in the proof of \cref{thm:lasserre-s+-iso} preserves non-negativity in both directions.

	\begin{theorem} \label{thm:lasserre-dnn-iso}
		Let $t \geq 1$. 
		Two graphs $G$ and $H$ are level-$t$ $\mathcal{DNN}$-isomorphic if and only if $G \simeq_t^{\textup{L}^+} H $. \end{theorem}

	\subsection{Isomorphisms between Matrix Algebras}
    \label{sec:matrix-algebras}

    To the two reformulations of $\simeq_t^{\textup{L}}$ and $\simeq_t^{\textup{L}^+}$ from the previous sections, a third characterisation is added in this section.
    It is shown that two graphs are level-$t$ $\mathcal{PSD}$-isomorphic ($\mathcal{DNN}$-isomorphic) if and only if certain matrix algebras associated to them are isomorphic.
    These algebras will be identified as the algebras of homomorphism tensors for graphs from the families $\mathcal{L}_t$ and $\mathcal{L}_t^+$.
    The so-called (partially) coherent algebras considered in this section are natural generalisations of the coherent algebras which are well-studied in the context of the $2$-dimensional Weisfeiler--Leman algorithm \cite{chen_lectures_2018}.

	\subsubsection{Partially Coherent Algebras and \texorpdfstring{$\mathcal{PSD}$}{PSD}-Isomorphism Maps}

	Let $S \subseteq \mathbb{C}^{n^t \times n^t}$. A matrix algebra $\mathcal{A} \subseteq \mathbb{C}^{n^t \times n^t}$ is \emph{$S$-partially coherent} if it is unital, self-adjoint, contains the all-ones matrix, and is closed under Schur products with any matrix in $S$.
    A matrix algebra $\mathcal{A} \subseteq \mathbb{C}^{n^t \times n^t}$ is \emph{self-symmetrical} if for every $A \in \mathcal{A}$ and $\sigma \in \mathfrak{S}_{2t}$ also $A^\sigma \in \mathcal{A}$.
    Note that for $t=1$, an algebra $\mathcal{A}$ is self-symmetrical if for all $A \in \mathcal{A}$ also $A^T \in \mathcal{A}$ where $A^T$ is the transpose of $A$.
	
	\begin{definition} \label{def:pce}
		Given a graph $G$, define its \emph{$t$-partially coherent algebra} $\widehat{\mathcal{A}}^t_G$ as the minimal self-symmetrical $S$-partially coherent algebra where $S$ is the set of homomorphism tensors of $(t,t)$-bilabelled atomic graphs for $G$.
		
		Two $n$-vertex graphs $G$ and $H$ are \emph{partially $t$-equivalent} if there is a \emph{partial $t$-equivalence}, i.e.\  a vector space isomorphism $\phi \colon \widehat{\mathcal{A}}^t_G \to\widehat{\mathcal{A}}^t_H$ such that
		\begin{enumerate}
			\item $\phi(M^*) = \phi(M)^*$ for all $M \in \widehat{\mathcal{A}}^t_G$,
			\item $\phi(MN) = \phi(M)\phi(N)$ for all $M, N \in \widehat{\mathcal{A}}^t_G$,
			\item $\phi(I) = I$, $\phi(\boldsymbol{A}_G) = \boldsymbol{A}_H$ for all $\boldsymbol{A} \in \mathcal{A}_t$, and $\phi(J) = J$,
			\item $\phi(\boldsymbol{A}_G \odot M) = \boldsymbol{A}_H \odot \phi(M)$ for all $\boldsymbol{A} \in \mathcal{A}_t$ and any $M \in \widehat{\mathcal{A}}^t_G$.
			\item $\phi(M^\sigma) = \phi(M)^\sigma$ for all $M \in \widehat{\mathcal{A}}^t_G$ and all $\sigma \in \mathfrak{S}_{2t}$.
		\end{enumerate}
	\end{definition}
 
	The following \cref{thm:partially-equivalent} extends \cite[Theorem~6.2]{mancinska_graph_2020}.

	\begin{theorem} \label{thm:partially-equivalent}
		Let $t \geq 1$.
		Two graphs $G$ and $H$ are partially $t$-equivalent
		if and only if
		there is a level-$t$ $\mathcal{PSD}$-isomorphism map from $G$ to $H$.
	\end{theorem}
	\begin{proof}Let $\Phi \colon \mathbb{C}^{V(G)^t \times V(G)^t} \to \mathbb{C}^{V(H)^t \times V(H)^t}$ be a level-$t$ $\mathcal{PSD}$-isomorphism map from $G$ to $H$, i.e.\  it satisfies \cref{i1,i2,i3,i4}.
		By \cref{rem:num-vertices,i2,i3}, $\Phi(\boldsymbol{A}_G) = \boldsymbol{A}_H$ for all atomic $\boldsymbol{A} \in \mathcal{A}_t$ and $\vert V(G) \rvert = \lvert V(H) \rvert \eqqcolon n$.
		Similarly, $\Phi^*(\boldsymbol{A}_H) = \boldsymbol{A}_G$ for all atomic $\boldsymbol{A}$ by \cref{thm4.7}.
		By \cref{i1,i2}, $\Phi$ is completely positive and unital.
		By \cref{thm4.7}, $\Phi^*(I) = I$ and thus $\Phi$ is trace-preserving  \cite[Lemma~4.2]{mancinska_graph_2020}.
		Furthermore,
		\[
			\Phi(\boldsymbol{A}_G) = \boldsymbol{A}_H, \quad 
			\Phi^*(\boldsymbol{A}_H) = \boldsymbol{A}_G, \quad 
			\Phi(J) = J = \Phi^*(J).
		\]
		for all atomic $\boldsymbol{A} \in \mathcal{A}_t$.
		Thus, \cref{lem:mg-4.9} implies that for any $W \in \mathbb{C}^{V(G)^t\times V(G)^t}$ we have
		$\Phi(\boldsymbol{A}_GW) = \boldsymbol{A}_H\Phi(W)$ and $\Phi(W\boldsymbol{A}_G) = \Phi(W)\boldsymbol{A}_H$ for  all atomic $\boldsymbol{A} \in \mathcal{A}_t$.
		Hence, the restriction of $\Phi$ to $\widehat{\mathcal{A}}^t_G$ witnesses that $G$ and $H$ are partially $t$-equivalent.
		
		Conversely, suppose that $\phi \colon \widehat{\mathcal{A}}^t_G \to\widehat{\mathcal{A}}^t_H$ is as in \cref{def:pce}.
        By \cite[Lemma~5.3]{mancinska_graph_2020}, $\phi$ is trace-preserving.
        By \cref{lem:mg-A.1}, there exists a unitary matrix $U \in \mathbb{C}^{n^t \times n^t}$ such that $\phi(X) = UXU^*$ for all $X \in \widehat{\mathcal{A}}^t_G$.
		Let $\widehat{\phi} \colon \mathbb{C}^{V(G)^t \times V(G)^t} \to \mathbb{C}^{V(H)^t \times V(H)^t}$ be the map given by $\widehat{\phi}(X) = UXU^*$.
		Let $\Pi \colon \mathbb{C}^{V(G)^t \times V(G)^t} \to \widehat{\mathcal{A}}^t_G$ be the orthogonal projection onto $\widehat{\mathcal{A}}^t_G$.
		Define $\Phi \colon \mathbb{C}^{V(G)^t \times V(G)^t} \to \mathbb{C}^{V(H)^t \times V(H)^t}$ by $\Phi \coloneqq \widehat{\phi} \circ \Pi$.
		By \cite[Lemma~5.3]{mancinska_graph_2020}, $\widehat{\phi}$ is completely positive and trace-preserving. By \cite[Lemma~5.4]{mancinska_graph_2020}, so is $\Pi$ and hence their composition $\Phi$. Hence, \cref{i1} holds.
		
		Furthermore, $\Pi(J) = J$ and hence $\Phi(J) =J = \Phi^*(J)$. So $\Phi$ satisfies \cref{i3}.

		For \cref{i4}, consider the linear map $\Lambda_\sigma \colon X \mapsto X^\sigma$ for $\sigma \in \mathfrak{S}_{2t}$. Since $\widehat{\mathcal{A}}_G$ is closed under the action of $\mathfrak{S}_{2t}$, it holds that $\Lambda_\sigma \circ \Pi = \Pi \circ \Lambda_\sigma \circ \Pi $. Furthermore, $(\Lambda_{\sigma})^{*} = \Lambda_{\sigma^{-1}}$ and $\Pi$ is self-adjoint, i.e.\  $\Pi^* = \Pi$. Hence,
		\[
		\Pi \circ \Lambda_{\sigma}
		= \Pi^* \circ \Lambda_\sigma
		= (\Lambda_{\sigma^{-1}} \circ \Pi)^*
		= (\Pi \circ \Lambda_{\sigma^{-1}} \circ \Pi)^*
		= \Pi \circ \Lambda_{\sigma} \circ \Pi
		= \Lambda_\sigma \circ \Pi.
		\]
		So $\Pi$ and $\Lambda_\sigma$ commute. Hence,
		\[
			\Phi(X^\sigma) 
			= (\widehat{\phi} \circ \Pi \circ \Lambda_\sigma)(X)
			= (\widehat{\phi} \circ \Lambda_\sigma \circ \Pi )(X)
			= ((\widehat{\phi} \circ \Pi)(X))^\sigma
			= \Phi(X)^\sigma.
		\]
		\Cref{i2} follows similarly, cf.\  the proof of \cite[Theorem~6.2]{mancinska_graph_2020}.
	\end{proof}

	\subsubsection{Coherent Algebras and \texorpdfstring{$\mathcal{DNN}$}{DNN}-Isomorphism Maps}

     \begin{figure}
        \centering
        \begin{subfigure}[t]{.3\textwidth}
            \centering
            \includegraphics[page=5,scale=1.2]{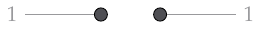}
            \caption{$\boldsymbol{J}$}
        \end{subfigure}
        \begin{subfigure}[t]{.3\textwidth}
            \centering
            \includegraphics[page=6,scale=1.2]{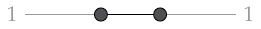}
            \caption{$\boldsymbol{A}^{1,2}$}
        \end{subfigure}
        \begin{subfigure}[t]{.3\textwidth}
            \centering
            \includegraphics[page=7,scale=1.2]{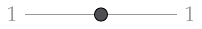}
            \caption{$\boldsymbol{I}^{1,2}$}
        \end{subfigure}
        \caption{The three atomic graphs in $\mathcal{A}_1$.}
        \label{fig:atomic-t1}
    \end{figure}

	A matrix algebra $\mathcal{A} \subseteq \mathbb{C}^{n^t\times n^t}$ is \emph{coherent} if it is unital, self-adjoint, contains the all-ones matrix and is closed under Schur products.

    For $t=1$, the $1$-adjacency algebra as defined below is equal to the well-studied \emph{adjacency algebra} of a graph $G$, cf.\  \cite{chen_lectures_2018}.
    The latter is the smallest coherent algebra containing the adjacency matrix of the graph.
    The former is generated by the homomorphism tensors of $(1,1)$-bilabelled atomic graphs.
    These graphs are depicted in \cref{fig:atomic-t1}. Their homomorphism tensors are the all-ones matrix, the adjacency matrix of the graph, and the identity matrix.
	
	\begin{definition} \label{def:ce}
		Let $t \geq 1$.
		The \emph{$t$-adjacency algebra} $\mathcal{A}_G^t$ of a graph $G$ is the self-symmetrical coherent algebra generated by the homomorphism tensors of the atomic graphs $\mathcal{A}_t$.
		
		Two $n$-vertex graphs $G$ and $H$ are \emph{$t$-equivalent} if there is \emph{$t$-equivalence}, i.e.\  a vector space isomorphism $\phi \colon \mathcal{A}^t_G \to \mathcal{A}^t_H$ such that
		\begin{enumerate}
			\item $\phi(M^*) = \phi(M)^*$ for all $M \in \mathcal{A}^t_G$,
			\item $\phi(MN) = \phi(M)\phi(N)$ for all $M, N \in \mathcal{A}^t_G$,
   		  \item $\phi(I) = I$, $\phi(\boldsymbol{A}_G) = \boldsymbol{A}_H$ for all $\boldsymbol{A} \in \mathcal{A}_t$, and $\phi(J) = J$,
			\item $\phi(M \odot N) = \phi(M) \odot \phi(N)$ for all $M, N \in \mathcal{A}^t_G$.
			\item $\phi(M^\sigma) = \phi(M)^\sigma$ for all $M \in \mathcal{A}^t_G$ and all $\sigma \in \mathfrak{S}_{2t}$.
		\end{enumerate}
	\end{definition}
	
	The following \cref{thm:equivalent} extends \cite[Theorem~7.3]{mancinska_graph_2020}.

	\begin{theorem} \label{thm:equivalent}
		Let $t \geq 1$.
		Two graphs $G$ and $H$ are $t$-equivalent
		if and only if
		there is a level-$t$ $\mathcal{DNN}$-isomorphism map from $G$ to $H$.
	\end{theorem}
    \begin{proof}Let $\Phi \colon \mathbb{C}^{V(G)^t \times V(G)^t} \to \mathbb{C}^{V(H)^t \times V(H)^t}$ be a level-$t$ $\mathcal{DNN}$-isomorphism map.
            Let $\phi$ be the restriction of $\Phi$ to $\mathcal{A}^t_G$.
    		Given the arguments in the proof of \cref{thm:partially-equivalent}, it suffices to show that $\phi(M \odot N) = \phi(M) \odot \phi(N)$  for all $M, N \in \mathcal{A}^t_G$ and that $\phi^*(M \odot N) = \phi^*(M) \odot \phi^*(N)$  for all $M, N \in \mathcal{A}^t_H$. This follows from \cite[Lemma~7.2]{mancinska_graph_2020}.
    		
    		Conversely, suppose that $\phi \colon \mathcal{A}^t_G \to \mathcal{A}^t_H$ is as in \cref{def:ce}.
            It follows as in \cite[Lemma~5.3]{mancinska_graph_2020} that $\phi$ is trace-preserving.
    		By \cref{lem:mg-A.1}, there exists a unitary matrix $U \in \mathbb{C}^{n^t \times n^t}$ such that $\phi(X) = UXU^*$ for all $X \in \mathcal{A}^t_G$.
    		Let $\widehat{\phi} \colon \mathbb{C}^{V(G)^t \times V(G)^t} \to \mathbb{C}^{V(H)^t \times V(H)^t}$ be the map given by $\widehat{\phi}(X) = UXU^*$.
    		Let $\Pi \colon \mathbb{C}^{V(G)^t \times V(G)^t} \to \mathcal{A}^t_G$ be the orthogonal projection onto $\mathcal{A}^t_G$.
    		Define $\Phi \colon \mathbb{C}^{V(G)^t \times V(G)^t} \to \mathbb{C}^{V(H)^t \times V(H)^t}$ by $\Phi \coloneqq \widehat{\phi} \circ \Pi$.
    		Given \cref{thm:partially-equivalent}, it suffices to argue that the Choi matrix of $\Phi$ is entry-wise non-negative. This can be done as in the proof of \cite[Theorem~7.3]{mancinska_graph_2020}.
    	\end{proof}

	\section{Homomorphism Indistinguishability}
	\label{sec:hi}
	
	Using techniques from \cite{grohe_homomorphism_2022}, we finally establish a characterisation of when the level-$t$ Lasserre relaxation of $\ISO(G, H)$ is feasible in terms of homomorphism indistinguishability of $G$ and $H$. In order to do so, we introduce the graph classes $\mathcal{L}_t$ and $\mathcal{L}_{t}^+$. 
	In \cref{sec:lt-ltplus,sec:further-relations}, we relate $\mathcal{L}_t$ and $\mathcal{L}_t^+$ to the classes of graphs of bounded treewidth and pathwidth obtaining the results depicted in \cref{fig:relationship}. In \cref{sec:t-equals-one}, $\mathcal{L}_1$ and $\mathcal{L}_1^+$ are identified as the classes of outerplanar graphs and graphs of treewidth two, respectively.

	\begin{definition} \label{def:l-lplus}
		Let $t \geq 1$.
		Write $\mathcal{L}_t^+$ for the class of $(t,t)$-bilabelled graphs generated by the set of atomic graphs $\mathcal{A}_t$ under parallel composition, series composition, and the action of $\mathfrak{S}_{2t}$ on the labels. 
        
        Write $\mathcal{L}_t \subseteq \mathcal{L}_t^+$ for the class of $(t,t)$-bilabelled graphs generated by the set of atomic graphs $\mathcal{A}_t$ under parallel composition with graphs from $\mathcal{A}_t$, series composition, and the action of $\mathfrak{S}_{2t}$ on the labels.
	\end{definition}

    Note that the only difference between $\mathcal{L}_t$ and $\mathcal{L}_t^+$ is that $\mathcal{L}_t$ is closed under parallel composition with atomic graphs only. This reflects an observation by \cite{grohe_homomorphism_2022} relating the closure under arbitrary gluing products to non-negative solutions to systems of equations characterising homomorphism indistinguishability. Intuitively, one may use arbitrary Schur products, the algebraic counterparts of gluing, for a Vandermonde interpolation argument, cf.\  \cite[Corollary~4.3]{grohe_homomorphism_2021_arxiv}.

    The following \cref{obs:lt-clique} illustrates how the operations in \cref{def:l-lplus} can be used to generate more complicated graphs from the atomic graphs, cf.\  \cref{fig:obs:lt-clique}.

    \begin{figure}
        \centering
        \includegraphics[page=10, scale=1.2]{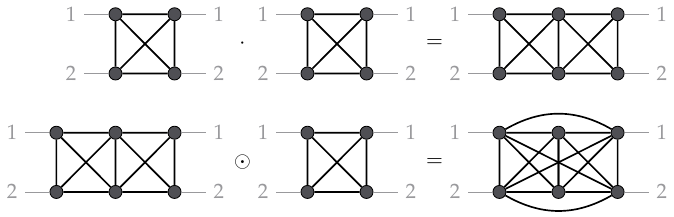}
        \caption{The bilabelled graphs in \cref{obs:lt-clique} for $t = 2$.}
        \label{fig:obs:lt-clique}
    \end{figure}
    \begin{observation} \label{obs:lt-clique}
		Let $t \geq 1$. The class $\mathcal{L}_t$ contains a bilabelled graph whose underlying unlabelled graph is isomorphic to the $3t$-clique $K_{3t}$.
	\end{observation}
	\begin{proof}
		Let $\boldsymbol{E} \coloneqq \bigodot_{1 \leq i < j \leq 2t} \boldsymbol{A}^{ij} \in \mathcal{A}_t$. The graph underlying $\boldsymbol{E} \odot (\boldsymbol{E} \cdot \boldsymbol{E})$ is isomorphic to $K_{3t}$.
	\end{proof}

    The only missing implications of \cref{thm:summary,thm:summary-pos} follow from the next two theorems:

    \begin{theorem} \label{thm:hi-l}
        Let $t \geq 1$. Two graphs $G$ and $H$ are homomorphism indistinguishable over~$\mathcal{L}_t$ if and only if they are partially $t$-equivalent.
    \end{theorem}
    \begin{theorem} \label{thm:hi-l-pos}
        Let $t \geq 1$. Two graphs $G$ and $H$ are homomorphism indistinguishable over~$\mathcal{L}_t^+$ if and only if they are $t$-equivalent.
    \end{theorem}

    For the proofs of \cref{thm:hi-l,thm:hi-l-pos}, we extend the framework developed by \cite{grohe_homomorphism_2022}.
    In this work, the authors introduced tools for constructing systems of equations characterising homomorphism indistinguishably over classes of labelled graphs.
    A requirement of these tools is that the graph class in question is \emph{inner-product compatible} \cite[Definition 24]{grohe_homomorphism_2022}.
    This means that for every two labelled graphs $\boldsymbol{R}$ and $\boldsymbol{S}$ one can write the inner-product of their homomorphism vectors $\boldsymbol{R}_G$ and $\boldsymbol{S}_G$ as the sum-of-entries of some $\boldsymbol{T}_G$ where $\boldsymbol{T}$ is labelled graph from the class. 
    Due to the correspondence between combinatorial operations on labelled graphs and algebraic operations on their homomorphism vectors, cf.\  \cref{sec:bilabelled}, this is equivalent to the graph theoretic assumption that $\soe(\boldsymbol{R} \odot \boldsymbol{S}) = \soe(\boldsymbol{T})$, i.e.\  the unlabelled graph obtained by unlabelling the gluing product of $\boldsymbol{R}$ and $\boldsymbol{S}$ can be labelled such that the resulting labelled graph is in the class.

    We extend this notion to bilabelled graphs.
	A class of $(t,t)$-bilabelled graphs $\mathcal{S}$ is said to be \emph{inner-product compatible} if for all $\boldsymbol{R}, \boldsymbol{S} \in \mathcal{S}$ there is a graph $\boldsymbol{T} \in \mathcal{S}$ such that $\tr(\boldsymbol{R} \cdot \boldsymbol{S}^*) = \soe(\boldsymbol{T})$.
	This definition is inspired by the inner-product on $\mathbb{C}^{n\times n}$ given by $\left< A, B \right> \coloneqq \tr(AB^*)$.

	\begin{lemma} \label{lem:ipc}
		Let $t \geq 1$. The classes $\mathcal{L}_t$ and $\mathcal{L}_t^+$ are inner-product compatible.
	\end{lemma}
	\begin{proof}
		Since  $\mathcal{L}_t$ is closed under matrix products and taking transposes, it suffices to show that for every $\boldsymbol{S} \in \mathcal{L}_t$ the graph $\tr(\boldsymbol{S})$ is the underlying unlabelled graph of some element of $\mathcal{L}_t$. Indeed, for every $(t,t)$-bilabelled graphs $\boldsymbol{F}$ it holds that 
		$\tr(\boldsymbol{F}) = \soe(\boldsymbol{I}^{1, t+1} \odot \dots \odot \boldsymbol{I}^{t, 2t} \odot \boldsymbol{F})$ where the $\boldsymbol{I}^{ij}$ are as in \cref{obs:atomic}. Since $\mathcal{L}_t$ is closed under parallel composition with atomic graphs, the claim follows. For $\mathcal{L}_t^+$, an analogous argument yields the claim.
	\end{proof}

    The following \cref{thm:ipc}, which extends the toolkit for constructing systems of equations characterising homomorphism indistinguishability over families of bilabelled graphs, is the bilabelled analogue of \cite[Theorem 13]{grohe_homomorphism_2022}.
    Write $\mathbb{C}\mathcal{S}_G \subseteq \mathbb{C}^{V(G)^t \times V(G)^t}$ for the vector space spanned by homomorphism tensors $\boldsymbol{S}_G$ for $\boldsymbol{S} \in \mathcal{S}$.

	\begin{theorem} \label{thm:ipc}
		Let $t \geq 1$ and $\mathcal{S}$ be an inner-product compatible class of $(t,t)$-bilabelled graphs containing~$\boldsymbol{J}$.
		For graphs $G$ and $H$, the following are equivalent:
		\begin{enumerate}
			\item $G$ and $H$ are homomorphism indistinguishable over $\mathcal{S}$,
			\item there exists a sum-preserving vector space isomorphism $\phi \colon \mathbb{C}\mathcal{S}_G \to \mathbb{C}\mathcal{S}_H$ such that $\phi(\boldsymbol{S}_G) = \boldsymbol{S}_H$ for all $\boldsymbol{S} \in \mathcal{S}$.
		\end{enumerate}
	\end{theorem}

    \begin{proof}For the forward direction, observe that for all $\boldsymbol{R}, \boldsymbol{S} \in \mathcal{S}$ it holds that $\left< \boldsymbol{R}_G, \boldsymbol{S}_G \right> = \tr(\boldsymbol{R}_G\boldsymbol{S}_G^*) = \left< \boldsymbol{R}_H, \boldsymbol{S}_H \right>$ by inner-product compatibility. Hence, by a Gram--Schmidt argument~\cite[Lemma~2.1]{grohe_homomorphism_2021_arxiv}, there exists a unitary map such that  $\phi(\boldsymbol{S}_G) = \boldsymbol{S}_H$ for all $\boldsymbol{S} \in \mathcal{S}$.
		Since $\phi(\boldsymbol{J}_G) = \boldsymbol{J}_H$ and $\phi$ is unitary, it is sum-preserving by \cref{rem:num-vertices}.
		Conversely, let $\phi$ be as stipulated. For every $\boldsymbol{S} \in \mathcal{S}$, it holds that $\soe(\boldsymbol{S}_H) = \soe(\phi(\boldsymbol{S}_G)) = \soe(\boldsymbol{S}_G)$ since $\phi$ is sum-preserving.
	\end{proof}

    This completes the preparations for the proof of \cref{thm:hi-l,thm:hi-l-pos}.

	\begin{proof}[Proof of Theorems \ref{thm:hi-l} and \ref{thm:hi-l-pos}]
        By comparing the operations from \cref{def:pce,def:l-lplus}, it follows that $\mathbb{C}\mathcal{S}_G = \widehat{\mathcal{A}}^t_G$ for $\mathcal{S} = \mathcal{L}_t$.
        By \cref{lem:ipc,thm:ipc}, $G$ and $H$ are homomorphism indistinguishable over $\mathcal{L}_t$ if and only if there is a sum-preserving vector space isomorphism $\phi \colon \widehat{\mathcal{A}}^t_G \to \widehat{\mathcal{A}}^t_H$ satisfying $\phi(\boldsymbol{S}_G) = \boldsymbol{S}_H$ for all $\boldsymbol{S} \in \mathcal{L}_t$.
		
		For all atomic $\boldsymbol{A} \in \mathcal{A}_t$, it holds that $\phi(\boldsymbol{A}_G) = \boldsymbol{A}_H$. Furthermore, since $\mathcal{L}_t$ is closed under the action of $\mathfrak{S}_{2t}$, $\phi(\boldsymbol{S}_G^\sigma) = \phi((\boldsymbol{S}^\sigma)_G) = (\boldsymbol{S}^\sigma)_H = \boldsymbol{S}_H^\sigma$ for all $\sigma \in \mathfrak{S}_{2t}$.
		Finally, for all $\boldsymbol{S}, \boldsymbol{T} \in \mathcal{L}_t$ it holds that $\phi(\boldsymbol{S}_G \cdot \boldsymbol{T}_G) = \phi((\boldsymbol{S} \cdot \boldsymbol{T})_G) = \boldsymbol{S}_H \cdot \boldsymbol{T}_H$ and $\phi(\boldsymbol{S}_G \odot \boldsymbol{T}_G) = \phi((\boldsymbol{S} \odot \boldsymbol{T})_G) = \boldsymbol{S}_H \odot \boldsymbol{T}_H$. 
        The homomorphism matrices $\boldsymbol{S}_G$ for $\boldsymbol{S} \in \mathcal{L}_t$ span  $\mathbb{C}\mathcal{S}_G = \widehat{\mathcal{A}}^t_G$. Hence, $\phi$ is a partial $t$-equivalence.
		
		Conversely, every partial $t$-equivalence $\phi \colon \widehat{\mathcal{A}}^t_G \to \widehat{\mathcal{A}}^t_H$ is such that $\phi(\boldsymbol{S}_G) = \boldsymbol{S}_H$ for all $\boldsymbol{S} \in \mathcal{L}_t$ by definition of $\mathcal{L}_t$.
		With slight modifications, \cite[Lemma~5.3]{mancinska_graph_2020} yields that $\phi$ is trace-preserving, which implies with $\phi(J) = J$ that $\phi$ is sum-preserving.
		The proof of \cref{thm:hi-l-pos} is analogous.
	\end{proof}

	\subsection{The Classes \texorpdfstring{$\mathcal{L}_t$}{Lt} and \texorpdfstring{$\mathcal{L}_t^+$}{Lt+} and Graphs of Bounded Treewidth}
	\label{sec:lt-ltplus}

	In this section, the classes $\mathcal{L}_t$ and $\mathcal{L}_t^+$ are compared to the classes of graphs of bounded treewidth and pathwidth. \Cref{fig:relationship} depicts the relationships between these classes. 
    The first result, \cref{lem:tw3t}, gives an upper bound on the treewidth of graphs in $\mathcal{L}_t^+$.

	\begin{lemma} \label{lem:tw3t}
		Let $t \geq 1$.
		The treewidth of an unlabelled graph $F$ underlying some  $\boldsymbol{F} = (F, \boldsymbol{u}, \boldsymbol{v}) \in \mathcal{L}_t^+$ is at most $3t - 1$.
	\end{lemma}
    \begin{proof}By structural induction, it is shown that every $\boldsymbol{F} = (F, \boldsymbol{u}, \boldsymbol{v}) \in \mathcal{L}^+_t$ admits a tree decomposition $\beta \colon V(T) \to 2^{V(F)}$ of width at most $3t-1$ such that the labelled vertices $\boldsymbol{u}$ and $\boldsymbol{v}$ lie together in one bag, i.e.\  there exists $x \in V(T)$ such that $\{\boldsymbol{u}_1, \dots, \boldsymbol{u}_t, \boldsymbol{v}_1, \dots, \boldsymbol{v}_t\} \subseteq \beta(x)$.
		
		  In the base case, i.e.\ if $\boldsymbol{F} \in \mathcal{A}_t$, then $\boldsymbol{F}$ has at most $2t$ vertices, which can all be placed in the single bag of a tree decomposition over the singleton tree.
    
		For the inductive step, let $\boldsymbol{F} = (F, \boldsymbol{u}, \boldsymbol{v})$ and $\boldsymbol{F}' = (F', \boldsymbol{u}', \boldsymbol{v}')$ from $\mathcal{L}^+_t$ be given. Suppose there are tree decompositions $\beta \colon V(T) \to 2^{V(F)}$ and $\beta' \colon V(T') \to 2^{V(F')}$ as in the inductive hypothesis. Let $x \in V(T)$ and  $x' \in V(T')$ be such that the labelled vertices of $\boldsymbol{F}$ and $\boldsymbol{F}'$ lie in $\beta(x)$ and $\beta'(x')$ respectively.
		Let $S$ be the tree obtained by taking the disjoint union of $T$, $T'$, and a fresh vertex $y$, and connecting $x$ and $x'$ to $y$. 
		
		For the graph $\boldsymbol{F} \cdot \boldsymbol{F}'$, an $S$-decomposition is given by the function
		\[
			\gamma \colon z \mapsto \begin{cases}
				\beta(z), & \text{if } z \in V(T), \\
				\beta'(z), & \text{if } z \in V(T'), \\
				\{\boldsymbol{u}_1, \dots, \boldsymbol{u}_t, \boldsymbol{v}'_1, \dots, \boldsymbol{v}'_t, \boldsymbol{v}_1, \dots, \boldsymbol{v}_t \},& \text{if } z = y.
			\end{cases}
		\]
		where one may note that $\boldsymbol{v}_i = \boldsymbol{u}'_i$ for every $i \in [t]$ in $\boldsymbol{F} \cdot \boldsymbol{F}'$. It is easy to check that \cref{def:definition} is satisfied. The decomposition is of width $3t -1$.
		
		For the graph $\boldsymbol{F} \odot \boldsymbol{F}'$, an $S$-decomposition is given by the function
		\[
		\gamma \colon z \mapsto \begin{cases}
			\beta(z), & \text{if } z \in V(T), \\
			\beta'(z), & \text{if } z \in V(T'), \\
			\{\boldsymbol{u}_1, \dots, \boldsymbol{u}_t,\boldsymbol{v}_1, \dots, \boldsymbol{v}_t \},& \text{if } z = y.
		\end{cases}
		\]
		where one may note that $\boldsymbol{u}_i = \boldsymbol{u}'_i$ and $\boldsymbol{v}_i = \boldsymbol{v}'_i$ for every $i \in [t]$ in $\boldsymbol{F} \odot \boldsymbol{F}'$. Again, it is easy to check that \cref{def:definition} is satisfied. The decomposition is of width at most $3t -1$.
	\end{proof}

    \Cref{lem:tw3t} in conjunction with \cref{thm:summary,thm:summary-pos} implies \cref{thm:main2,thm:main3}.
    As a corollary, this yields the upper bound in \cref{thm:main}.
    Indeed, by \cref{thm:sa}, $G \simeq_t^{\textup{SA}} H$ if and only if $G$ and $H$ are homomorphism indistinguishable over the class of graphs of treewidth at most $t-1$. Hence, if $G \simeq_{3t}^{\textup{SA}} H$ then $G \simeq_t^{\textup{L}^+} H$ and in particular $G \simeq_t^{\textup{L}} H$.

    It remains to show the lower bound asserted by \cref{thm:main}, i.e.\  that $3t$ cannot be replaced by $3t-1$ for no $t \geq 1$.
    To that end, first observe that \cref{obs:lt-clique} implies that the bound in \cref{lem:tw3t} is tight.
    However, this syntactic property of the graph class $\mathcal{L}_t$ does not suffice to derive the aforementioned semantic property of $\simeq_{t}^{\textup{SA}}$ and $\simeq_{t}^{\textup{L}}$.
    In fact, it could well be that for all graphs $G$ and $H$ if $G$ and $H$ are homomorphism indistinguishable over the graphs of treewidth at most $3t-2$ also $\hom(K_{3t}, G) = \hom(K_{3t}, H)$ despite that $\tw K_{3t} > 3t-2$.
    That this does not hold is implied by a conjecture of the first author \cite{roberson_oddomorphisms_2022} which asserts that every minor-closed graph class $\mathcal{F}$ which is closed under taking disjoint unions (\emph{union-closed}) is \emph{homomorphism distinguishing closed}, i.e.\  for all $F \not\in \mathcal{F}$ there exist graphs $G$ and $H$ such that $G \equiv_{\mathcal{F}} H$ but $\hom(F, G) \neq \hom(F, H)$.
    Although being generally open, this conjecture was proven by Neuen~\cite{neuen_homomorphism-distinguishing_2023} for the class of graphs of treewidth at most $t$ for every $t$. \Cref{thm:lower-bound} implies the last assertion of \cref{thm:main}.
 
	\begin{theorem} \label{thm:lower-bound}
        For every $t \geq 1$, there exist graphs $G$ and $H$ such that $G \simeq_{3t-1}^{\textup{SA}} H$ and $G \not\simeq_{t}^{\textup{L}} H$.
	\end{theorem}
    \begin{proof}
        Since $\tw(K_{3t}) = 3t-1$, there exist, by \cite[Theorem~2]{neuen_homomorphism-distinguishing_2023}, two graphs $G$ and $H$ such that $G \equiv_{\mathcal{TW}_{3t-2}} H$ and $\hom(K_{3t}, G) \neq \hom(K_{3t}, H)$.
        By \cref{thm:sa}, $G \simeq_{3t-1}^{\textup{SA}} H$.
        By \cref{obs:lt-clique,thm:summary},  $G \not\simeq_{t}^{\textup{L}} H$.
    \end{proof}

    \subsection{Bilabelled Minors}
    
    It is worth noting that the classes of unlabelled graphs underlying the elements of $\mathcal{L}_t$ and $\mathcal{L}_t^+$ are themselves minor-closed and union-closed. Hence, they are subject to the aforementioned conjecture. Furthermore, by the Robertson--Seymour Theorem and \cite{robertson_XIII_1995}, membership in $\mathcal{L}_t$ and $\mathcal{L}_t^+$ can be tested in polynomial time for every fixed $t \geq 1$.  
    \begin{lemma} \label{lem:minor-closed}
        Let $t \geq 1$.
        The class of graphs underlying the elements of $\mathcal{L}_t$ and the class of graphs underlying the elements of $\mathcal{L}_t^+$ are minor-closed and union-closed.
    \end{lemma}

    In order to proof \cref{lem:minor-closed}, we introduce bilabelled analogues of graph minors. 
    The tools developed here will also be used in \cref{sec:t-equals-one}.
    
    \begin{definition}\label{def:bminor}
    Let $\boldsymbol{M}$ and $\boldsymbol{F}$ be $(\ell, k)$-bilabelled graphs for some $k, \ell \in \mathbb{N}$.
    Then $\boldsymbol{M}$ is a \emph{bilabelled minor} of $\boldsymbol{F}$, in symbols $\boldsymbol{M} \leq \boldsymbol{F}$, if it can be obtained from $\boldsymbol{F}$ by applying a sequence of the following \emph{bilabelled minor operations}:
    \begin{enumerate}
    \item edge contraction,\label{bm1}
    \item edge deletion,\label{bm2}
    \item deletion of unlabelled vertices,\label{bm3}
    \end{enumerate}
    A family of bilabelled graphs $\mathcal{F}$ is \emph{minor-closed} if it is closed under taking bilabelled minors.
    \end{definition}

    Note that for $(0,0)$-bilabelled graphs, i.e.\  unlabelled graphs, \cref{def:bminor} and the standard definition of graph minors coincide.

    \begin{example} \label{ex:atomic-minor}
        Let $t\geq 1$. The class of atomic graphs $\mathcal{A}_t$ as defined in \cref{def:atomic} is minor-closed.
    \end{example}

    We proceed to prove various lemmas characterising how bilabelled minors behave under the operations applied to bilabelled graphs, namely labelling and unlabelling and series and parallel composition.
    
    \begin{lemma}[Minor Unlabelling Lemma]\label{lem:mul}
    Let $\boldsymbol{M} \leq \boldsymbol{F}$ be bilabelled.
    Then $\down{\boldsymbol{M}} \leq \down{\boldsymbol{F}}$.
    \end{lemma}
    \begin{proof}
    It is argued by induction on the number of bilabelled minor operations necessary to transform $\boldsymbol{F}$ into $\boldsymbol{M}$. If $\boldsymbol{M} = \boldsymbol{F}$ then $\soe(\boldsymbol{M}) = \soe(\boldsymbol{F})$, and the claim follows. Suppose that $\boldsymbol{M} \leq \boldsymbol{M}' \leq \boldsymbol{F}$ where $\boldsymbol{M}'$ can be transformed into $\boldsymbol{M}$ by applying a single minor operation and $\boldsymbol{M}'$ is minimal among all such graphs with respect to the number of minor operations necessary to derive it from $\boldsymbol{F}$.
    By the inductive hypothesis, $\down{\boldsymbol{M}'} \leq \down{\boldsymbol{F}}$.
    Since bilabelled minor operations are more restrictive than minor operations, any operation of \cref{def:bminor} carried out on $\boldsymbol{M}'$ can be applied to $\down{\boldsymbol{M}'}$. It follows that $\down{\boldsymbol{M}} \leq \down{\boldsymbol{F}}$.
    \end{proof}
    
    \begin{lemma}[Minor Labelling Lemma]\label{lem:mll}
    Let $\boldsymbol{F}$ be bilabelled and $M$ be unlabelled. 
    If $M \leq \down{\boldsymbol{F}}$ then there exists $\boldsymbol{M} \leq \boldsymbol{F}$ such that $\down{\boldsymbol{M}}$ is the disjoint union of $M$ and potential isolated vertices which are labelled in $\boldsymbol{M}$.
    \end{lemma}
    \begin{proof}
    It is argued by induction on the number of minor operations needed to transform $\down{\boldsymbol{F}}$ into $M$.
    If $M = \down{\boldsymbol{F}}$, let $\boldsymbol{M} \coloneqq \boldsymbol{F}$.
    Now suppose there are $M \leq M' \leq \down{\boldsymbol{F}}$ such that $M'$ can be transformed into $M$ by applying a single minor operation. Then there exists $\boldsymbol{M}' \leq \boldsymbol{F}$ such that $\down{\boldsymbol{M}'}$ is the disjoint union of $M'$ and potential isolated vertices. Distinguish cases:
    \begin{itemize}
    \item $M$ is obtained from $M'$ by deleting or contracting an edge $e$. Then $e$ has a counterpart in $\boldsymbol{M}'$ since $\down{\boldsymbol{M}'}$ contains $M'$. Contracting/deleting the edge there yields the desired~$\boldsymbol{M}$. 
    \item $M$ is obtained from $M'$ by deleting a vertex $v$. If $v$ is unlabelled in $\boldsymbol{M}'$ then it can be deleted from $\boldsymbol{M}'$ yielding $\boldsymbol{M}$. If $v$ is labelled in $\boldsymbol{M}'$, remove all edges incident to $v$ and let $\boldsymbol{M}$ be the resulting graph. In this case, $\down{\boldsymbol{M}}$ is the disjoint union of $M$ and an isolated vertex. \qedhere
    \end{itemize}
    \end{proof}

    Intuitively, the following \cref{lem:mpl,lem:mppl} assert that minor operations commute with bilabelled graph multiplication.

    \begin{lemma}[Minor Parallel Composition Lemma]\label{lem:mppl}
        Let $\boldsymbol{P}_1$ and $\boldsymbol{P}_2$ be $(k,\ell)$-bilabelled graphs.
        \begin{enumerate}
            \item If $\boldsymbol{M}_1$ is a minor of $\boldsymbol{P}_1$ and $\boldsymbol{M}_2$ is a minor of $\boldsymbol{P}_2$ then $\boldsymbol{M}_1\odot \boldsymbol{M}_2$ is a minor of $\boldsymbol{P}_1 \odot \boldsymbol{P}_2$.
            \item If $\boldsymbol{K}$ is a minor of $\boldsymbol{P}_1 \odot \boldsymbol{P}_2$ then there exist $(k, \ell)$-bilabelled $\boldsymbol{M}_1$ and $\boldsymbol{M}_2$ such that
                $\boldsymbol{K} = \boldsymbol{M}_1 \odot \boldsymbol{M}_2$,
                $\boldsymbol{M}_1$ is a minor of $\boldsymbol{P}_1$, and
                $\boldsymbol{M}_2$ is a minor of $\boldsymbol{P}_2$.
        \end{enumerate}
    \end{lemma}
    \begin{proof}
        For the first claim, it is argued by induction on the sum of the number of minor operations applied to transform $\boldsymbol{P}_1$ into $\boldsymbol{M}_1$ and $\boldsymbol{P}_2$ into $\boldsymbol{M}_2$. For the base case, $\boldsymbol{M}_1 = \boldsymbol{P}_1$ and $\boldsymbol{M}_2 = \boldsymbol{P}_2$, and the claim follows trivially.
        
        Now suppose that $\boldsymbol{M}_1$ is obtained from $\boldsymbol{M}'_1$, a minor of $\boldsymbol{P}_1$, by applying a single minor operation. Suppose inductively that $\boldsymbol{M}'_1 \odot \boldsymbol{M}_2$ is a minor of $\boldsymbol{P}_1 \odot \boldsymbol{P}_2$. Distinguish cases:
        \begin{itemize}
        \item $\boldsymbol{M}_1$ is obtained from $\boldsymbol{M}'_1$ by contracting an edge $e$. In $\boldsymbol{M}'_1 \odot \boldsymbol{M}_2$, this edge is either a loop or a proper edge. In the former case, it can be deleted, in the latter case, it can be contracted, yielding in both cases $\boldsymbol{M}_1 \odot \boldsymbol{M}_2$.
        \item $\boldsymbol{M}_1$ is obtained from $\boldsymbol{M}'_1$ by deleting an edge $e$. In $\boldsymbol{M}'_1 \odot \boldsymbol{M}_2$, this edge is either a loop or a proper edge. In both cases, it can be deleted yielding $\boldsymbol{M}_1 \odot \boldsymbol{M}_2$.
        \item $\boldsymbol{M}_1$ is obtained from $\boldsymbol{M}'_1$ by deleting an unlabelled vertex~$v$. Then $v$ is unlabelled in $\boldsymbol{M}'_1 \odot \boldsymbol{M}_2$ and can be deleted. The resulting graph is $\boldsymbol{M}_1 \odot \boldsymbol{M}_2$.
        \end{itemize}

        For the second claim, it is argued by induction on the number of minor operations necessary to transform $\boldsymbol{P}_1 \odot \boldsymbol{P}_2$ into $\boldsymbol{K}$. 
        For the base case, if $\boldsymbol{K} = \boldsymbol{P}_1 \odot \boldsymbol{P}_2$, let $\boldsymbol{M} \coloneqq \boldsymbol{K}$, $\boldsymbol{M}_1 \coloneqq \boldsymbol{P}_1$, and $\boldsymbol{M}_2 \coloneqq \boldsymbol{P}_2$.
    
        Now suppose that $\boldsymbol{K}$ is a minor of $\boldsymbol{P}_1 \odot \boldsymbol{P}_2$. 
        Then there exists a $(k,\ell)$-bilabelled graph  $\boldsymbol{K}'$ such that $\boldsymbol{K}'$ is a minor of $\boldsymbol{P}_1 \odot \boldsymbol{P}_2$ and $\boldsymbol{K}$ is obtained from $\boldsymbol{K}'$ by applying a single minor operation. 
        By the induction hypothesis, there exist $\boldsymbol{M}'_1$ and $\boldsymbol{M}'_2$ such that the assertions of this lemma are satisfied.
        Distinguish cases:
        \begin{itemize}
        \item $\boldsymbol{K}$ is obtained from $\boldsymbol{K}'$ by deleting or contracting an edge~$e$.
        
        The edge $e$ may lie in both $\boldsymbol{M}'_1$ and $\boldsymbol{M}'_2$ or  in only one of the two graphs.
In either case, 
        construct $\boldsymbol{M}_1$ and $\boldsymbol{M}_2$ by respectively deleting or contracting the edge in $\boldsymbol{M}'_1$ and $\boldsymbol{M}'_2$ or leaving the graph unchanged if it does not contain the edge.
        
        \item $\boldsymbol{K}$ is obtained from $\boldsymbol{K}'$ by deleting an unlabelled vertex~$v$.

        Since no vertex is unlabelled under parallel composition, the vertex $v$ is also unlabelled in the graph $\boldsymbol{M}'_1$ or $\boldsymbol{M}'_2$ which it contains.
        It follows that $v$ can be deleted from $\boldsymbol{M}'_i$ leaving the other graph untouched. This yields $\boldsymbol{M}_1$ and $\boldsymbol{M}_2$.
        \qedhere
        \end{itemize}
    \end{proof}

    \begin{lemma}[Minor Series Composition Lemma]\label{lem:mpl}
    Let $\boldsymbol{P}_1$ be $(k,\ell)$-bilabelled and $\boldsymbol{P}_2$ be $(\ell, j)$-bilabelled.
    \begin{enumerate}
    \item If $\boldsymbol{M}_1$ is a minor of $\boldsymbol{P}_1$ and $\boldsymbol{M}_2$ is a minor of $\boldsymbol{P}_2$ then $\boldsymbol{M}_1\cdot \boldsymbol{M}_2$ is a minor of $\boldsymbol{P}_1 \cdot \boldsymbol{P}_2$.
    \item If $\boldsymbol{K}$ is a minor of $\boldsymbol{P}_1 \cdot \boldsymbol{P}_2$ then 
    there exists a $(k, j)$-bilabelled $\boldsymbol{M}$, a $(k,\ell)$-bilabelled $\boldsymbol{M}_1$, and a $(\ell, j)$-bilabelled $\boldsymbol{M}_2$ such that 
    	\begin{enumerate}
    		\item $\boldsymbol{M}$ is the disjoint union of $\boldsymbol{K}$ and potential isolated unlabelled vertices, which are labelled in $\boldsymbol{M}_1$ and $\boldsymbol{M}_2$,\label{mplA}
    		\item $\boldsymbol{M} = \boldsymbol{M}_1 \cdot \boldsymbol{M}_2$, and\label{mplB}
    		\item $\boldsymbol{M}_1$ is a minor of $\boldsymbol{P}_1$
                and $\boldsymbol{M}_2$ is a minor of $\boldsymbol{P}_2$.\label{mplC}
    	\end{enumerate}
    \end{enumerate}
    \end{lemma}

    \begin{proof}
    The proof of the first claim is analogous to the proof of the first claim of \cref{lem:mppl}.
    
    For the second claim, it is argued by induction on the number of minor operations necessary to transform $\boldsymbol{P}_1 \cdot \boldsymbol{P}_2$ into $\boldsymbol{K}$. For the base case, if $\boldsymbol{K} = \boldsymbol{P}_1 \cdot \boldsymbol{P}_2$, let $\boldsymbol{M} \coloneqq \boldsymbol{K}$, $\boldsymbol{M}_1 \coloneqq \boldsymbol{P}_1$, and $\boldsymbol{M}_2 \coloneqq \boldsymbol{P}_2$.
    
    Now suppose that $\boldsymbol{K}$ is a minor of $\boldsymbol{P}_1 \cdot \boldsymbol{P}_2$. Then there exists a $(k,j)$-bilabelled graph  $\boldsymbol{K}'$ such that $\boldsymbol{K}'$ is a minor of $\boldsymbol{P}_1 \cdot \boldsymbol{P}_2$ and $\boldsymbol{K}$ is obtained from $\boldsymbol{K}'$ by applying a single minor operation. 
    By the induction hypothesis, there exist $\boldsymbol{M}'$, $\boldsymbol{M}'_1$, and $\boldsymbol{M}'_2$ such that \cref{mplA,mplB,mplC} are satisfied.
    Distinguish cases:
    \begin{itemize}
    \item $\boldsymbol{K}$ is obtained from $\boldsymbol{K}'$ by deleting or contracting an edge~$e$.
    
    Define $\boldsymbol{M}$ by deleting/contracting the same edge in $\boldsymbol{M}'$.
    The edge $e$ may lie in both $\boldsymbol{M}'_1$ and $\boldsymbol{M}'_2$ or only in one of the two graphs.
    In the first case, both endpoints of $e$ are labelled in both graphs. In either case, 
    construct $\boldsymbol{M}_1$ and $\boldsymbol{M}_2$ by respectively deleting or contracting the edge in $\boldsymbol{M}'_1$ and $\boldsymbol{M}'_2$ or leaving the graph unchanged if it does not contain the edge.
    
    \item $\boldsymbol{K}$ is obtained from $\boldsymbol{K}'$ by deleting an unlabelled vertex~$v$.
    
    If $v$ is among the unlabelled vertices of $\boldsymbol{M}'_i$ for $i\in \{1,2\}$ then define $\boldsymbol{M}$ by deleting $v$ from $\boldsymbol{M}'$.
    It follows that $v$ can be deleted from $\boldsymbol{M}'_i$ leaving the other graph untouched. This yields $\boldsymbol{M}_1$ and $\boldsymbol{M}_2$.
    
    If otherwise $v$ is among the vertices at which $\boldsymbol{M}'_1$  and $\boldsymbol{M}'_2$ are glued together then 
    define $\boldsymbol{M}$ as the graph obtained from $\boldsymbol{M}'$ by deleting all edges incident with $v$ but keeping the vertex.
    By the inductive hypothesis, $\boldsymbol{M}'$ is the disjoint union of $\boldsymbol{K}'$ and isolated unlabelled vertices and via the aforementioned construction the same holds for $\boldsymbol{M}$ and $\boldsymbol{K}$.
    Note that $v$ is neither in-labelled in $\boldsymbol{M}'_1$ nor out-labelled in $\boldsymbol{M}'_2$ as it would otherwise be labelled in~$\boldsymbol{M}$.
    Delete all edges incident to $v$ in both $\boldsymbol{M}'_1$ and $\boldsymbol{M}'_2$. The resulting $\boldsymbol{M}_1$ and $\boldsymbol{M}_2$ satisfy $\boldsymbol{M} = \boldsymbol{M}_1 \cdot \boldsymbol{M}_2$, as desired.
    \qedhere
    \end{itemize}
    \end{proof}

    With these general facts at hand, we proceed to show the following about our graph classes $\mathcal{L}_t$ and $\mathcal{L}_t^+$:
    \begin{lemma} \label{cl:bilabelled-minor-closed}
            Let $t \geq 1$.
            The classes $\mathcal{L}_t$ and $\mathcal{L}_t^+$ are closed under taking bilabelled minors.
        \end{lemma}
        \begin{proof}
            By induction on the structure of elements $\boldsymbol{F} \in \mathcal{L}_t$, it is proven that if $\boldsymbol{K} \leq \boldsymbol{F}$ then also $\boldsymbol{K} \in \mathcal{L}_t$.
            For $\mathcal{L}_t^+$, the proof is very similar, requiring fewer case distinctions. It is therefore omitted.
            If $\boldsymbol{F}$ is atomic then all its minors are atomic by \cref{ex:atomic-minor}. This constitutes the base case of the induction.
            
            If $\boldsymbol{F} = \boldsymbol{F}_1 \odot \boldsymbol{F}_2$ for $\boldsymbol{F}_1 \in \mathcal{A}_t$, $\boldsymbol{F}_2 \in \mathcal{L}_t$ to which the inductive hypothesis applies, and $\boldsymbol{K} \leq \boldsymbol{F}$ then, by \cref{lem:mppl}, there exist $\boldsymbol{K}_1 \leq \boldsymbol{F}_1$ and $\boldsymbol{K}_2 \leq \boldsymbol{F}_2$ such that $\boldsymbol{K} = \boldsymbol{K}_1 \odot \boldsymbol{K}_2$. 
            By \cref{ex:atomic-minor}, $\boldsymbol{K}_1$ is atomic and, by the inductive hypothesis, $\boldsymbol{K}_2 \in \mathcal{L}_t$. Hence, $\boldsymbol{K} \in \mathcal{L}_t$.
    
            If $\boldsymbol{F} = \boldsymbol{F}_1 \cdot \boldsymbol{F}_2$ for two $\boldsymbol{F}_1, \boldsymbol{F}_2 \in \mathcal{L}_t$ to which the inductive hypothesis applies and $\boldsymbol{K} \leq \boldsymbol{F}$ then, by \cref{lem:mpl}, there exist $\boldsymbol{M}$, $\boldsymbol{M}_1$, $\boldsymbol{M}_2$ such that $\boldsymbol{M}_1 \leq \boldsymbol{F}_1$, $\boldsymbol{M}_2 \leq \boldsymbol{F}_2$, and $\boldsymbol{M} = \boldsymbol{M}_1 \cdot \boldsymbol{M}_2$ is the disjoint union of $\boldsymbol{K}$ and potential isolated unlabelled vertices which are labelled both in $\boldsymbol{M}_1$ and $\boldsymbol{M}_2$.
            By the inductive hypothesis, $\boldsymbol{M}_1, \boldsymbol{M}_2 \in \mathcal{L}_t$.
            It remains to remove these isolated vertices. Suppose that the $i$-th out-label of $\boldsymbol{M}_1$ and the $i$-th in-label of $\boldsymbol{M}_2$ are carried by an isolated vertex. Then graph $(\boldsymbol{I}^{1,t+1+i} \odot \boldsymbol{M}_1) \cdot \boldsymbol{M}_2$ does not contain this isolated vertex since taking the parallel composition with $\boldsymbol{I}^{1,t+1+i}$ as defined in \cref{def:atomic} amounts to gluing it to the first in-labelled vertex of $\boldsymbol{M}_1$. Observe that $\boldsymbol{I}^{1,t+1+i} \odot \boldsymbol{M}_1 \in \mathcal{L}_t$.
            Proceeding in this fashion, one can construct $\boldsymbol{K}_1, \boldsymbol{K}_2 \in \mathcal{L}_t$ such that $\boldsymbol{K} = \boldsymbol{K}_1 \cdot \boldsymbol{K}_2 \in \mathcal{L}_t$, as desired.

            If $\boldsymbol{F} = \boldsymbol{F}_1^\sigma$ for $\sigma \in \mathfrak{S}_{2t}$, $\boldsymbol{F}_1 \in \mathcal{L}_t$ to which the inductive hypothesis applies, and $\boldsymbol{K} \leq \boldsymbol{F}$ then $\boldsymbol{K}^{\sigma^{-1}} \leq \boldsymbol{F}_1$ and $\boldsymbol{K}^{\sigma^{-1}} \in \mathcal{L}_t$ by the inductive hypothesis. Hence, $\boldsymbol{K} \in \mathcal{L}_t$, as desired.
        \end{proof}

    This concludes the preparations for the proof of \cref{lem:minor-closed}.

    \begin{proof}[Proof of Lemma \ref{lem:minor-closed}]
        For $(t,t)$-bilabelled graphs  $\boldsymbol{F}$ and $\boldsymbol{F}'$ and $\boldsymbol{J} \in \mathcal{A}_t$ as defined in \cref{def:atomic}, the graph underlying $\boldsymbol{F} \cdot \boldsymbol{J} \cdot \boldsymbol{F}'$  is isomorphic to the disjoint union of the graphs underlying $\boldsymbol{F}$ and $\boldsymbol{F}'$. Hence, the classes of graphs underlying elements of $\mathcal{L}_t$ and $\mathcal{L}_t^+$ are union-closed.

        Given \cref{cl:bilabelled-minor-closed}, it remains to observe that the classes of unlabelled graphs underlying the elements of $\mathcal{L}_t$ and $\mathcal{L}_t^+$ are minor closed.
        By \cref{lem:mll}, if an unlabelled graph $M$ is a minor of $\soe(\boldsymbol{F})$ for some $\boldsymbol{F} \in \mathcal{L}_t$ then there exists $\boldsymbol{M} \leq \boldsymbol{F}$ such that $\soe(\boldsymbol{M})$ is the disjoint union of $M$ and potential isolated vertices which are labelled in $\boldsymbol{M}$. By \cref{cl:bilabelled-minor-closed}, $\boldsymbol{M} \in \mathcal{L}_t$. As in the proof of \cref{cl:bilabelled-minor-closed}, the potential isolated vertices can be identified with other labelled in $\boldsymbol{M}$ by taking the parallel composition of this graph with atomic graphs. Hence, it may be assumed that $M = \soe(\boldsymbol{M})$. This yields the claim.
    \end{proof}

    \subsection{Further Relations between \texorpdfstring{$\mathcal{TW}_t$, $\mathcal{PW}_t$, $\mathcal{L}_t$, and $\mathcal{L}_t^+$}{Graph Classes}}
    \label{sec:further-relations}
    
    This subsection is dedicated to some further relations between the classes of graphs of bounded treewidth or pathwidth, $\mathcal{L}_t$, and $\mathcal{L}_t^+$. 
    These facts give independent proofs for the correspondence between the feasibility of the level-$t$ Sherali--Adams relaxation (without non-negativity constraints), which corresponds to homomorphism indistinguishability over graphs of treewidth (pathwidth) at most $t-1$, as proven by \cite{dell_lovasz_2018,grohe_homomorphism_2022}, and the feasibility of the level-$t$ Lasserre relaxation with and without non-negativity constraints.

	First of all,
    dropping the semidefiniteness constraint \cref{lassere1} of the level-$t$ Lasserre system of equations turns this system essentially into the level-$2t$ Sherali--Adams system of equations without non-negativity constraints, e.g.\  as defined in \cite[Section~2.7]{grohe_homomorphism_2021_arxiv}. This is paralleled by \cref{lem:pw}.
    
	\begin{lemma} \label{lem:pw}
		Let $t \geq 1$.
		For every graph $F$ with $\pw F \leq 2t - 1$,
		there is a graph $\boldsymbol{F} \in \mathcal{L}_t$ whose underlying unlabelled graph is isomorphic to $F$.
	\end{lemma}
     \begin{proof}If $|V(F)| \leq 2t$ then there exists an atomic graph $\boldsymbol{F} \in \mathcal{A}_t$ whose underlying unlabelled graph is isomorphic to $F$. Otherwise, by \cref{lem:bodlaender8pw}, there exists a path decomposition $\beta \colon V(P) \to 2^{V(F)}$ such that  $|\beta(v)| = 2t$ for all $v \in V(P)$ and $|\beta(s) \cap \beta(t)| = 2t-1$ for all $st \in E(P)$.
		
		It is shown by induction on $|V(P)|$ that for every vertex $r \in V(P)$ of degree at most one there exist $\boldsymbol{u} = u_1\dots u_t \in V(F)^t$, $\boldsymbol{v} = v_1\dots v_t \in V(F)^t$ with $\beta(r) = \{u_1,\dots, u_t, v_1, \dots, v_t\}$ such that $\boldsymbol{F} = (F, \boldsymbol{u},\boldsymbol{v}) \in \mathcal{L}_t$. 
		
		The inductive argument is very similar to the one in the proof of \cref{lem:tw2t}.
		Indeed, since the vertex~$r$ has at most one neighbour, $\ell \leq 1$ in the proof of \cref{lem:tw2t} and the construction does not require arbitrary parallel compositions.
	\end{proof}

    Furthermore, one may drop \cref{lassere1} from the level-$t$ Lasserre system of equations with non-negativity constraints to obtain the level-$2t$ Sherali--Adams system of equations in its original form, i.e.\  with non-negativity constraints. This is paralleled by \cref{lem:tw2t}.

    \begin{lemma} \label{lem:tw2t}
		Let $t \geq 1$.
		For every graph $F$ with $\tw F \leq 2t-1$,
		there is a graph $\boldsymbol{F} \in \mathcal{L}_t^+$ whose underlying unlabelled graph is isomorphic to $F$.
	\end{lemma}
     \begin{proof}If $|V(F)| \leq 2t$ then there exists an atomic graph $\boldsymbol{F} \in \mathcal{A}_t$ whose underlying unlabelled graph is isomorphic to $F$. Otherwise, by \cref{lem:bodlaender8}, there exists a tree decomposition $\beta \colon V(T) \to 2^{V(F)}$ of $F$ such that $|\beta(v)| = 2t$ for all $v \in V(T)$ and $|\beta(s) \cap \beta(t)| = 2t-1$ for all $st \in E(T)$. It is shown by induction on $|V(T)|$ that for every $r \in V(T)$ there exist $\boldsymbol{u} = u_1\dots u_t \in V(F)^t$, $\boldsymbol{v} = v_1\dots v_t \in V(F)^t$ with $\beta(r) = \{u_1,\dots, u_t, v_1, \dots, v_t\}$ such that $\boldsymbol{F} = (F, \boldsymbol{u},\boldsymbol{v}) \in \mathcal{L}_t^+$. 
    		Observe that this implies that the labels of $\boldsymbol{F}$ lie on distinct vertices of $F$.
    		
    		In the base case, when $|V(T)| = 1$, the tuples  $\boldsymbol{u}$ and $\boldsymbol{v}$ can be chosen arbitrarily subject to the desired condition and $\boldsymbol{F}$ is an atomic graph.
    		
    		Let $|V(T)| \geq 2$ and $r \in V(T)$ be arbitrary. 
            Write $s_1, \dots, s_\ell$ for the neighbours of $r$ in $T$.
    		First a bilabelled graph $\boldsymbol{F}_i \in \mathcal{L}_t^+$ is constructed for each $i \in [\ell]$.
    		Let $T_i$ be the connected component of $T \setminus \{r\}$ containing $s_i$.
    		Let $F_i$ be the induced subgraph of $F$ on $\bigcup_{t \in V(T_i)} \beta(t)$. 
    		The restriction of $\beta$ to $V(T_i)$ is a tree decomposition of $F_i$ with the properties stated in the inductive hypothesis. 
    		Hence, there exist $\boldsymbol{u}^i = u^i_1\dots u^i_t \in V(F_i)^t$, $\boldsymbol{v}^i = v^i_1\dots v^i_t \in V(F_i)^t$ with $\beta(s_i) = \{u^i_1,\dots, u^i_t, v^i_1, \dots, v^i_t\}$ such that $\boldsymbol{F}_i \coloneqq (F_i, \boldsymbol{u}^i, \boldsymbol{v}^i) \in \mathcal{L}_t^+$.
            
           	Let $x_1, \dots, x_{2t}$ denote the vertices in $\beta(r)$. By permuting labels, it can be guaranteed that for every $i \in [\ell]$, the tuples $u^i_1\dots u^i_tv^i_1\dots v^i_t$ and $x_1 \dots x_{2t}$ differ at precisely one index $j_i \in [2t]$.
    		Recall the bilabelled graphs defined in \cref{obs:atomic} and $\boldsymbol{K}^j$ from \cref{eq:kj,fig:kj}.
        	Let $\boldsymbol{F}'_i \coloneqq \boldsymbol{K}^{j_i} \cdot \boldsymbol{F}_i$ if $j_i \leq t$ and  $\boldsymbol{F}'_i \coloneqq \boldsymbol{F}_i \cdot \boldsymbol{K}^{j_i - t}$ otherwise.
        	Intuitively, the bilabelled graph $\boldsymbol{F}'_i$ is obtained from $\boldsymbol{F}_i$ by adding a fresh vertex and moving the $j_i$-th label to this vertex.
            Since $\boldsymbol{F}_i \in \mathcal{L}_t^+$ and $\boldsymbol{K}^{j_i} \in \mathcal{A}_t$,
            it holds that $\boldsymbol{F}_i' \in \mathcal{L}_t^+$.
        	Finally, let $\boldsymbol{F} = \boldsymbol{F}'_1 \odot \dots \odot \boldsymbol{F}'_\ell \odot\bigodot_{x_ix_j \in E(F)} \boldsymbol{A}^{ij}$.
    	\end{proof}

    Since the diagonal entries of a positive semidefinite matrix are necessarily non-negative, \cref{lassere1} implies that any solution $(y_I)$ to the level-$t$ Lasserre system of equations is such that $y_I \geq 0$ for all $I \in \binom{V(G) \times V(H)}{\leq t}$. Hence, such a solution is a solution to the level-$t$ Sherali--Adams system of equations as well. 
    This is paralleled by \cref{lem:twt}.

	\begin{lemma} \label{lem:twt}
		Let $t \geq 1$.
		For every graph $F$ with $\tw F \leq t-1$,
		there is a graph $\boldsymbol{F} \in \mathcal{L}_t$ whose underlying unlabelled graph is isomorphic to $F$.
	\end{lemma}
 \begin{proof}If $|V(F)| \leq t$ then there exists an atomic graph $\boldsymbol{F} \in \mathcal{A}_t$ whose underlying unlabelled graph is isomorphic to $F$. Otherwise, by \cref{lem:bodlaender8}, there exists a tree decomposition $\beta \colon V(T) \to 2^{V(F)}$ of $F$ such that $|\beta(v)| = t$ for all $v \in V(T)$ and $|\beta(s) \cap \beta(t)| = t-1$ for all $st \in E(T)$. It is shown by induction on $|V(T)|$ that for every $r \in V(T)$ there exist $\boldsymbol{u} = u_1\dots u_t \in V(F)^t$ with $\beta(r) = \{u_1,\dots, u_t\}$ such that $\boldsymbol{F} = (F, \boldsymbol{u},\boldsymbol{u}) \in \mathcal{L}_t$.
		
		In the base case, when $|V(T)| = 1$, the tuple $\boldsymbol{u}$ can be chosen arbitrarily  and $\boldsymbol{F}$ is an atomic graph.
		
		Let $|V(T)| \geq 2$ and $r \in V(T)$ be arbitrary. 
		Write $s_1, \dots, s_\ell$ for the neighbours of $r$ in $T$.
		First a graph $\boldsymbol{F}_i \in \mathcal{L}_t$ is constructed for each $i \in [\ell]$.
		Let $T_i$ be the connected component of $T \setminus \{r\}$ containing $s_i$.
		Let $F_i$ be the induced subgraph of $F$ on $\bigcup_{t \in V(T_i)} \beta(t)$. 
		The restriction of $\beta$ to $V(T_i)$ is a tree decomposition of $F_i$ with the properties listed in the inductive hypothesis. 
		Hence, there exist $\boldsymbol{u}^i = u^i_1\dots u^i_t \in V(F_i)^t$ with $\beta(s_i) = \{u^i_1,\dots, u^i_t\}$ such that $\boldsymbol{F}_i \coloneqq (F_i, \boldsymbol{u}^i, \boldsymbol{u}^i) \in \mathcal{L}_t$.
		
		Let $x_1, \dots, x_{t}$ denote the vertices in $\beta(r)$. By permuting labels, it can be guaranteed that for every $i \in [\ell]$, the tuples $u^i_1\dots u^i_t$ and $x_1 \dots x_{t}$ differ at precisely one index $j_i \in [t]$.
		Recall the bilabelled graphs defined in \cref{obs:atomic} and $\boldsymbol{K}_j$ from \cref{eq:kj,fig:kj}. 
        Let $\boldsymbol{F}_i' \coloneqq \boldsymbol{I}^{j_i, t+j_i} \odot (\boldsymbol{K}_{j_i} \cdot \boldsymbol{F} \cdot \boldsymbol{K}_{j_i})$.
		By construction, $\boldsymbol{F}_i' \in \mathcal{L}_t$. The labelled vertices of $\boldsymbol{F}_i$ differ from those of $\boldsymbol{F}_i$ in $x_{j_i}$. 
		Finally, let
		\[
			\boldsymbol{F} \coloneqq (\boldsymbol{I}^{1, t+1} \odot \dots \odot \boldsymbol{I}^{t, 2t}) \odot (\boldsymbol{F}_1' \cdot \dots \cdot \boldsymbol{F}_\ell')  \odot \bigodot_{x_ix_j \in E(F)} \boldsymbol{A}^{ij}.
		\]
		This graph is as desired.
	\end{proof}

	\subsection{The Classes \texorpdfstring{$\mathcal{L}_1$}{L1} and \texorpdfstring{$\mathcal{L}_1^+$}{L1+}}
	\label{sec:t-equals-one}

    The classes $\mathcal{L}_1$ and $\mathcal{L}_1^+$ can be identified as the class of outerplanar graphs and as the class of graphs of treewidth at most two, respectively.
    This yields \cref{thm:main4}.

	\begin{theorem} \label{prop:l1plus}
		The class of unlabelled graphs underlying an element of $\mathcal{L}_1^+$ coincides with the class of graphs of treewidth at most two.
	\end{theorem}
    \begin{proof}Given \cref{lem:tw3t,lem:tw2t}, it suffices to show that if a graph $F$ is such that $\tw F = 2$ then there is a graph $\boldsymbol{F} \in \mathcal{L}_1^+$ whose underlying unlabelled graph is isomorphic to $F$. 
		
		By \cref{lem:bodlaender8}, there exists a tree decomposition $\beta \colon V(T) \to 2^{V(F)}$ of $F$ such that $|\beta(v)| = 3$ for all $v \in V(T)$ and $|\beta(s) \cap \beta(t)| = 2$ for all $st \in E(T)$. It is shown by induction on $|V(T)|$ that for every $r \in V(T)$ and $x \neq y \in \beta(r)$ the graph $\boldsymbol{F} = (F, x, y)$ is in $\mathcal{L}_1^+$.
		
		If $|V(T)| = 1$, write $\{x, y, z\}$ for the unique bag. Since $F$ has treewidth $2$, it is isomorphic to the $3$-clique which is the underlying unlabelled graph of $\boldsymbol{A}^{12} \odot (\boldsymbol{A}^{12} \cdot \boldsymbol{A}^{12})$, cf.\  \cref{obs:lt-clique}, which is contained in~$\mathcal{L}_1^+$ by construction.
		
		Assuming $|V(T)| \geq 2$, let $r \in V(T)$ be arbitrary.
		Write $\beta(r) = \{x_1, x_2, x_3\}$.
		Partition the neighbours of $r$ in $T$ in three sets $X_1, X_2, X_3$ such that $s \in X_i$ iff $x_i \in \beta(r) \setminus \beta(s)$ for $i \in [3]$.
		
		For every neighbour $s$ of $r$, let $T_s$ be the connected component of $T \setminus \{r\}$ containing $s$.
		Let $F_s$ be the induced subgraph of $F$ on $\bigcup_{t \in V(T_s)} \beta(t)$. 
		The restriction of $\beta$ to $V(T_s)$ is a tree decomposition of $F_s$ with the properties listed in the inductive hypothesis. 
		Hence, for every $s$, there exists $\boldsymbol{F}_s \in \mathcal{L}_1^+$ as stipulated. By permuting labels, it may be supposed that for every $s \in X_1$ the labels of $\boldsymbol{F}_s$ lie on $x_2x_3$, for $\boldsymbol{F}_s$ with $s \in X_2$ on $x_1x_3$, and for $\boldsymbol{F}_s$ with $s \in X_3$ on $x_1x_2$.
		For $i \in [3]$, let 
        \[
            \boldsymbol{F}_i \coloneqq \begin{cases}
                \bigodot_{s \in X_i} \boldsymbol{F}_s, & \text{if } X_i \neq \emptyset, \\
                \boldsymbol{A}^{12}, & \text{if } X_i = \emptyset \text{ and the two vertices in  } \beta(r) \setminus \{x_i\} \text{ are adjacent},\\
                \boldsymbol{J}, & \text{otherwise.}
            \end{cases}
        \]
		Finally, let $\boldsymbol{F} \coloneqq \boldsymbol{F}_2 \odot (\boldsymbol{F}_3 \cdot \boldsymbol{F}_1)$. This graph is as desired if $x_1,x_3$ are required to be labelled. For other choices of labels, $\boldsymbol{F}_1, \boldsymbol{F}_2, \boldsymbol{F}_3$ can be permuted and if necessary transposed yielding any desired labelling. 
	\end{proof}

	A graph $F$ is \emph{outerplanar} if it does not have $K_4$ or $K_{2,3}$ as a minor. Equivalent, it is outerplanar if it has a planar drawing such that all its vertices lie on the same face~\cite{syslo_characterisations_1979}.
	
	\begin{theorem} \label{lem:op}
		The class of unlabelled graphs underlying an element of $\mathcal{L}_1$ coincides with the class of outerplanar graphs.
	\end{theorem}

    Before we prove \cref{lem:op}, we derive the following corollary:
     
	\begin{corollary}
        \label{lem:connected}
		If $G \equiv_{\mathcal{L}_1} H$ then $G$ is connected iff $H$ is connected.
	\end{corollary}
    \begin{proof}Let $\boldsymbol{D} = (D, v, v)$ be the $(1,1)$-bilabelled graph with $V(D) = \{u, v\}$ and $E(D) = \{uv\}$ and write $\boldsymbol{A}$ as before for the $(1,1)$-bilabelled graph corresponding to the adjacency matrix. For every graph $G$, the homomorphism matrix $\boldsymbol{D}_G - \boldsymbol{A}_G$ equals its Laplacian matrix. By~\cite[Lemma~4]{van_dam_which_2003}, if two graphs $G$ and $H$ have cospectral Laplacians then they have the same number of connected components.
		The former condition holds iff $\tr\left( (\boldsymbol{D}_G - \boldsymbol{A}_G)^i \right) = \tr\left((\boldsymbol{D}_H - \boldsymbol{A}_H)^i\right)$ for all $i \in \mathbb{N}$ by Newton's identities \cite{dawar_descriptive_2019}. The bilabelled graphs appearing as summands in the expression $\tr\left((\boldsymbol{D} - \boldsymbol{A})^i\right)$ are cactus graphs and hence outerplanar. By \cref{lem:op}, if $G \equiv_{\mathcal{L}_1} H$ then $G$ and $H$ have cospectral Laplacians and hence the same number of connected components.
	\end{proof}

    Towards proving \cref{lem:op}, we define a class of $(1,1)$-bilabelled graphs whose underlying unlabelled graphs are outerplanar.
    In general, carefully imposing conditions on where the labels are placed is essential for ensuring that the class of bilabelled graphs is closed under the desired operations and also generated by atomic graphs under them, cf.\  \cite[p.\ ~2271]{rattan_weisfeiler_2023}.
    
	\begin{definition} \label{def:op}
        The \emph{expansion} of a $(1,1)$-bilabelled graphs $\boldsymbol{F} = (F, u, v)$ is the graph $F'$ obtained from $F$ by adding a path of length two between $u$ and $v$, i.e.\  $V(F') \coloneqq V(F) \sqcup \{x\}$ and $E(F') \coloneqq E(F) \sqcup \{ux, xv\}$.
		Write $\mathcal{OP}$ for the class of $(1,1)$-bilabelled graphs $\boldsymbol{F}$ whose expansion is outerplanar.
	\end{definition}

    Note that the above definition implies that for all $\boldsymbol{F} = (F, u, v) \in \mathcal{OP}$ the underlying unlabelled graph $F$ is outerplanar as it is a minor of the expansion of $\boldsymbol{F}$.
    If the two labels of $\boldsymbol{F}$ coincide then its expansion is obtained by adding a dangling edge and outerplanar iff the underlying unlabelled graph of $\boldsymbol{F}$ is outerplanar.

	Write $\boldsymbol{A}$ and $\boldsymbol{I}$ for the $(1,1)$-bilabelled graphs corresponding to the adjacency matrix and the identity matrix respectively.
	In the notation of \cref{obs:atomic}, $\boldsymbol{A} = \boldsymbol{A}^{12}$ and $\boldsymbol{I} = \boldsymbol{I}^{12}$. These graphs are depicted in \cref{fig:atomic-t1}.

	\begin{lemma} \label{lem:closure}
		The class $\mathcal{OP}$ possesses the following closure properties:
		\begin{enumerate}
			\item If $\boldsymbol{F} \in \mathcal{OP}$ then $\boldsymbol{F}^* \in \mathcal{OP}$.
			\item If $\boldsymbol{F} \in \mathcal{OP}$ then $\boldsymbol{A} \odot \boldsymbol{F} \in \mathcal{OP}$ and $\boldsymbol{I} \odot \boldsymbol{F} \in \mathcal{OP}$.
			\item If $\boldsymbol{F}_1,\boldsymbol{F}_2 \in \mathcal{OP}$ then $\boldsymbol{F}_1 \cdot \boldsymbol{F}_2 \in \mathcal{OP}$.
		\end{enumerate}
	\end{lemma}
	\begin{proof}
		The first claim is purely syntactical. The underlying unlabelled graphs of $\boldsymbol{F}$ and $\boldsymbol{F}^*$ are isomorphic and so are their expansions. Thus, $\boldsymbol{F}^* \in \mathcal{OP}$ if $\boldsymbol{F} \in \mathcal{OP}$.

        For the second claim, first consider the case when the labels of $\boldsymbol{F}$ coincide. 
        Then $\boldsymbol{I} \odot \boldsymbol{F} = \boldsymbol{F}$ and $\boldsymbol{A} \odot \boldsymbol{F}$ differs from $\boldsymbol{F}$ only in the loop at the labelled vertex. 
        Hence, $\boldsymbol{A} \odot \boldsymbol{F}$ is in $\mathcal{OP}$. 
        Now consider the case when the labelled vertices of $\boldsymbol{F}$ are distinct.
        Write $F$ for the unlabelled graph underlying $\boldsymbol{F}$ and $F'$ for the expansion of $\boldsymbol{F}$.
		It can be easily seen that the graphs underlying $\boldsymbol{A} \odot \boldsymbol{F}$ and $\boldsymbol{I} \odot \boldsymbol{F}$ are minors of $F'$ and thus outerplanar.
		The expansion of $\boldsymbol{I} \odot \boldsymbol{F}$ is a minor of the expansion of $\boldsymbol{A} \odot \boldsymbol{F}$. 
        Thus, it suffices to argue that the expansion of $\boldsymbol{A} \odot \boldsymbol{F}$ is outerplanar. 

        Write $K$ for the unlabelled graph underlying $\boldsymbol{A} \odot \boldsymbol{F}$ and $K'$ for the expansion of $\boldsymbol{A} \odot \boldsymbol{F}$.
        Since $K$ and $F'$ are outerplanar, any $K_4$-minor of $K'$ can be obtained from $K'$ without contracting the triangle induced by the labelled vertices of $\boldsymbol{A} \odot \boldsymbol{F}$ and the vertex added by expansion. This cannot be since the latter vertex is of degree two. 
        By the same argument, since $K_{2,3}$ is triangle-free, the graph $K'$ does not contain any $K_{2,3}$-minor either.
		
		For the third claim, let $F$ denote the graph underlying $\boldsymbol{F}_1 \cdot \boldsymbol{F}_2$. Let $y$ denote the vertex at which $\boldsymbol{F}_1$ and $\boldsymbol{F}_2$ are glued together and write $x, z$ for the vertices labelled in $\boldsymbol{F}_1 \cdot \boldsymbol{F}_2$. The graph $F - y$ is disconnected. Hence, if $K_4$ or $K_{2,3}$ is a minor of $F$ then $\boldsymbol{F}_1$ or $\boldsymbol{F}_2$ are not outerplanar. Hence, $F$ is outerplanar.

        \begin{figure}
            \centering
            \begin{subfigure}{.45\linewidth}
                \centering
                \includegraphics[page=8,scale=1.2]{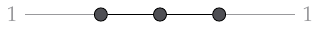}
                \caption{The graph $\boldsymbol{P} = \boldsymbol{A} \cdot \boldsymbol{A}$.}
                \label{fig:proof:closure1}
            \end{subfigure}
            \begin{subfigure}{.45\linewidth}
                \centering
                \includegraphics[page=9,scale=1.2]{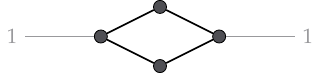}
                \caption{The graph $\boldsymbol{C}$.}
                \label{fig:proof:closure2}
            \end{subfigure}
            \caption{Bilabelled graphs from the proof of \cref{lem:closure}.}
            \label{fig:proof:closure}
        \end{figure}

        Write $F'$ for the expansion of $\boldsymbol{F} \coloneqq \boldsymbol{F}_1 \cdot \boldsymbol{F}_2$. In symbols, $F' = \soe(\boldsymbol{P} \odot \boldsymbol{F})$ where $\boldsymbol{P}$ is the bilabelled graph in \cref{fig:proof:closure1}.
        For $K \in \{K_4, K_{2,3} \}$, observe the following:
        If $F'$ contains $K$ as a minor then, by \cref{lem:mll}, there exists a bilabelled minor $\boldsymbol{K} \leq \boldsymbol{P} \odot \boldsymbol{F}$ such that $\soe(\boldsymbol{K})$ is the disjoint union of $K$ and potential isolated vertices which are labelled in $\boldsymbol{K}$. By \cref{lem:mppl}, $\boldsymbol{K}$ can be written as $\boldsymbol{K} = \boldsymbol{K}_1 \odot \boldsymbol{K}_2$ such that $\boldsymbol{K}_1 \leq \boldsymbol{P}$ and $\boldsymbol{K}_2 \leq \boldsymbol{F}$. The graph $\boldsymbol{P}$ has six bilabelled minors. Distinguish cases:
        \begin{enumerate}
            \item If $\boldsymbol{K}_1 = \boldsymbol{P}$ then $K = K_{2,3}$. The labels of $\boldsymbol{K}_2$ must lie on distinct vertices because $\boldsymbol{K}$ does not contain any vertices of degree one. Furthermore, the labelled vertices in $\boldsymbol{K}$ must be connected via a path of length two with an intermediate vertex of degree two. Hence, $\boldsymbol{K}_2 = \boldsymbol{C}$ where $\boldsymbol{C}$ is the graph in \cref{fig:proof:closure2}.\label{case:c1}
            \item If $\boldsymbol{K}_1 = \boldsymbol{A}$ then the labels of $\boldsymbol{K}_2$ must lie on distinct vertices because $\boldsymbol{K}$ does not contain any loops. Furthermore, the labelled vertices in $\boldsymbol{K}$ must be adjacent. Hence, $\boldsymbol{K}_2$ is a graph obtained from $K_4$ or $K_{2,3}$ by labelling two adjacent vertices and potentially removing the edge between them. In any case, $\boldsymbol{C} \leq \boldsymbol{K}_2$.\label{case:c2}
            \item If $\boldsymbol{K}_1 = \boldsymbol{I}$ then $\boldsymbol{K}_2$ is obtained from $K_4$ or $K_{2,3}$ by either picking one vertex and placing both labels on it or by adding a fresh vertex, placing a label on it, and connecting it to a subset of the neighbours of a chosen original vertex, which receives the other label.\label{case:identity}
            \item If $\boldsymbol{K}_1 = \boldsymbol{J}$ then $\boldsymbol{K}_2 = \boldsymbol{K}$. In particular, $\boldsymbol{K}_2$ is obtained from $K_4$ or $K_{2,3}$ by the procedure described in \cref{case:identity}.\label{case:disjoint}
            \item $\boldsymbol{K}_1$ cannot be any of the two remaining bilabelled minors of $\boldsymbol{P}$ since these contain an unlabelled vertex of degree at most one which is not the case for $\boldsymbol{K}$.
        \end{enumerate}
        For \cref{case:c1,case:c2} when $\boldsymbol{C} \leq \boldsymbol{F} = \boldsymbol{F}_1 \cdot \boldsymbol{F}_2$, then, by \cref{lem:mpl}, $\boldsymbol{C} \leq \boldsymbol{F}_1$ or $\boldsymbol{C} \leq \boldsymbol{F}_2$ because the graph $\boldsymbol{C}$ cannot be written as the series composition of two graphs different from $\boldsymbol{I}$. The bilabelled minor $\boldsymbol{C}$ of $\boldsymbol{F}_1$ or $\boldsymbol{F}_2$ gives rise to a $K_{2,3}$-minor in their expansion, contradicting that $\boldsymbol{F}_1, \boldsymbol{F}_2 \in \mathcal{OP}$.

        For \cref{case:identity,case:disjoint}, let $\boldsymbol{K}_2$ be the graph described there. 
        This graph can only be written as the series composition of two graphs different from $\boldsymbol{I}$ if the two labels do not coincide. In this case, one of the labelled vertices is adjacent to a subset of neighbours of the other labelled vertex. The graph $\boldsymbol{K}_2$ may be written as series composition of $\boldsymbol{A}$ or $\boldsymbol{J}$ with another graph $\boldsymbol{K}'_2$. The graph $\soe(\boldsymbol{K}'_2)$ contains $K_4$ or $K_{2,3}$ as a minor. By \cref{lem:mul}, one of the factors $\boldsymbol{F}_1$ or $\boldsymbol{F}_2$ is not outerplanar, a contradiction.
	\end{proof}

    We proceed to prove the following auxiliary lemma.
    For a vertex $u$ of a graph $F$, 
    write $N_F(u) \coloneqq \{v \in V(F) \mid uv \in E(F) \}$ for the set of neighbours of~$u$.
    
	\begin{lemma} \label{lem:op-helper}
		Let $F$ be an outerplanar graph with vertex $u \in V(F)$. If $u$ is not isolated then there exists a neighbour $v \in N_F(u)$ such that the graph obtained from $F$ by subdividing the edge $uv$ is outerplanar.
	\end{lemma}
	\begin{proof}
        Take an outerplanar embedding of $F$ which has some face incident to all the vertices, consider some edge incident to $u$ that is incident to this face, and subdivide that edge. 
        Since the vertex created by subdivision is incident to the outer face,
	       the embedding remains outerplanar when the edge is subdivided.
        Alternatively, one may consider the following argument:
 
		If $u$ is of degree one or two, then any of its neighbours is as desired.
		If $u$ has degree at least three, observe that $F[N_F(u)]$ cannot contain $K_3$ as a minor. Indeed, any such minor would give rise to a $K_4$-minor in $F$. Hence, $F[N_F(u)]$ is a forest and contains a vertex $v$ of degree at most one in $F[N_F(u)]$. Write $F'$ for the graph obtained from $F$ by subdividing the edge $uv$. Write $w$ for the vertex added this way.

        If $F'$ is not outerplanar then it contains a minor $K_4$ or $K_{2,3}$ which can be obtained from $F'$ without undoing the subdivision. This minor cannot be $K_4$ because $w$ is of degree two. Hence, $F'$ contains a $K_{2,3}$-minor which can be obtained from $F$ such that the path $uwv$ is not contracted. This implies that $v$ is adjacent to at least two neighbours of $u$ which cannot be since it was chosen to be of degree one in $F[N_F(u)]$, a contraction. The graph $F'$ is outerplanar.
	\end{proof}

    \Cref{lem:op-helper} facilitates decomposing bilabelled outerplanar graphs into simpler ones.

	\begin{lemma} \label{lem:generation}
		Let $\boldsymbol{F} = (F, u, v) \in \mathcal{OP}$ have $n \geq 3$ vertices.
		\begin{enumerate}
\item If $u = v$ then $\boldsymbol{F} = \boldsymbol{I} \odot (\boldsymbol{K} \cdot \boldsymbol{J})$ 
			or $\boldsymbol{F} = \boldsymbol{I} \odot ((\boldsymbol{A} \odot \boldsymbol{K}) \cdot \boldsymbol{J})$
			where $\boldsymbol{K} = (K, x, y) \in \mathcal{OP}$ has at most $n$ vertices and $x \neq y$.\label{gen2}
			\item If $uv \in E(F)$ then $\boldsymbol{F} = \boldsymbol{A} \odot \boldsymbol{K}$ where $\boldsymbol{K} = (K, x, y)\in \mathcal{OP}$ has at most $n$ vertices and satisfies $x \neq y$ and $xy \not\in E(K)$,\label{gen3}
			\item If $u \neq v$ and $uv \not\in E(F)$ then $\boldsymbol{F} = \boldsymbol{K} \cdot \boldsymbol{L}$ where $\boldsymbol{K},\boldsymbol{L}\in \mathcal{OP}$ have at least $2$ and at most $n-1$ vertices.\label{gen4}
		\end{enumerate}
	\end{lemma}
	\begin{proof}
		For \cref{gen2}, distinguishing two cases.
        \begin{itemize}
            \item If $u = v$ is isolated in $F$ then let $x \in V(F) \setminus \{u\}$ be arbitrary. Define $K \coloneqq F$. Then $\boldsymbol{K} \coloneqq (K, u, x)$ is such that $\boldsymbol{F} = \boldsymbol{I} \odot (\boldsymbol{K} \cdot \boldsymbol{J})$.
		      By definition, $K$ is outerplanar. 
            Since $u$ is isolated, the expansion of $\boldsymbol{K}$ differs from $\boldsymbol{K}$ only in the loop at $u$. Hence, $\boldsymbol{K}$ is outerplanar as well.
            
            \item If $u = v$ is not isolated, pick a neighbour $x$ in virtue of \cref{lem:op-helper}, and let $K$ be the graph obtained from $F$ by deleting the edge $ux$, i.e.\  $V(K) \coloneqq V(F)$ and $E(K) \coloneqq E(F) \setminus \{ux\}$. Let $\boldsymbol{K} \coloneqq (K, u, x)$. 
        As a subgraph of $F$, $K$ is outerplanar. The expansion of $\boldsymbol{K}$ is the graph obtained from $F$ by subdividing the edge $ux$ and outerplanar by \cref{lem:op-helper}. Hence, $\boldsymbol{K} \in \mathcal{OP}$. Furthermore, $\boldsymbol{F} = \boldsymbol{I} \odot ((\boldsymbol{A} \odot \boldsymbol{K}) \cdot \boldsymbol{J})$.
        \end{itemize}

		For \cref{gen3}, define $K$ by removing the edge $uv$ from $F$, i.e.\  $V(K) \coloneqq V(F)$ and $E(K) \coloneqq E(F) \setminus \{uv\}$. The graph $\boldsymbol{K} \coloneqq (K, u, v)$ satisfies $\boldsymbol{F} = \boldsymbol{A} \odot \boldsymbol{K}$ and all other stipulated properties.
		
		For \cref{gen4}, first suppose that $u$ and $v$ lie in the same connected component of $F$. 
Observe that there no two internally vertex-disjoint paths from $u$ to $v$ since a pair of two such paths would give rise to a $K_{2,3}$-minor in the expansion of $\boldsymbol{F}$.
        By Menger's Theorem, there exists a vertex $x \neq u, v$ meeting all paths from $u$ to $v$.
		Thus, removing $x$ from $F$ causes $u$ and $v$ to lie in separate connected components. 
        Let $A$ denote the connected component of $F - x$ containing $u$, $B$ the connected component of $F-x$ containing $v$, and $C$ the union of all connected components of $F-x$ containing neither $u$ nor $v$. 
        By definition, $V(F) = A \sqcup B \sqcup C \sqcup \{x\}$. 
        Define $K \coloneqq F[A \cup \{x\}]$ as the subgraph of $F$ induced by $A \cup \{x\}$ and similarly $L \coloneqq F[B \cup C \cup \{x\}]$. 
        Let $\boldsymbol{K} \coloneqq (K, u, x)$ and $\boldsymbol{L}  \coloneqq (L, x, v)$. Then $\boldsymbol{F} = \boldsymbol{K} \cdot \boldsymbol{L}$, as desired.
		As they are induced subgraphs of $F$, the graphs $K$ and $L$ are outerplanar. The expansions of $\boldsymbol{K}$ and $\boldsymbol{L}$ are minors of the expansion of $\boldsymbol{F}$ and thus outerplanar.
		Observe that $|V(K)| + |V(L)| = n+1$ and $|V(K)|, |V(L)| \geq 2$, as desired.

		Now suppose that $u$ and $v$ lie in separate connected components of $F$. 
        Let $A$ denote the connected component of $F$ containing $u$, $B$ the connected component of $F$ containing~$v$, and $C$ the union of all connected components of $F$ containing neither $u$ nor $v$. 
        Observe that $|A| + |B| + |C| = n \geq 3$.
        Distinguish cases:
        \begin{itemize}
            \item If $|A| + |C| \geq 2$, 
            let $K \coloneqq F[A \cup C]$ and $L' \coloneqq F[B]$.
            Define $\boldsymbol{K} \coloneqq (K, u, u)$, $\boldsymbol{L}' \coloneqq (L', v, v)$, and  $\boldsymbol{L} \coloneqq \boldsymbol{J} \cdot \boldsymbol{L}'$.
            \item Otherwise, it holds that $|B| \geq 2$.
            Let $K' \coloneqq F[A \cup B]$ and $L \coloneqq F[B]$.
            Define $\boldsymbol{K}' \coloneqq (K', u, u)$, $\boldsymbol{L} \coloneqq (L, v, v)$, and $\boldsymbol{K} \coloneqq \boldsymbol{K}' \cdot \boldsymbol{J}$.
        \end{itemize}
        In both cases, $\boldsymbol{F} = \boldsymbol{K} \cdot \boldsymbol{L}$ and $\boldsymbol{K}, \boldsymbol{L} \in \mathcal{OP}$.
        Furthermore, writing $K$ and $L$ for the graphs underlying $\boldsymbol{K}$ and $\boldsymbol{L}$ respectively, 
        it holds that $|V(K)| + |V(L)| = n+1$ and $|V(K)|, |V(L)| \geq 2$ since multiplication with $\boldsymbol{J}$ amounts to adding a fresh isolated vertex.
	\end{proof}
	
	The following \cref{cor:op} implies \cref{lem:op}.
	
	\begin{theorem} \label{cor:op}
		The classes $\mathcal{L}_1$ and $\mathcal{OP}$ coincide.
	\end{theorem}
	\begin{proof}
		For the inclusion $\mathcal{L}_1 \subseteq \mathcal{OP}$, observe that the atomic graphs in $\mathcal{A}_1$ are $\boldsymbol{A}, \boldsymbol{J}, \boldsymbol{I} \in \mathcal{OP}$, cf.\  \cref{fig:atomic-t1}.
		By \cref{lem:closure}, $\mathcal{OP}$ is closed under series composition, parallel composition with atomic graphs, and permutation of labels. It follows inductively that $\mathcal{L}_1 \subseteq \mathcal{OP}$.
		
		For the inclusion $\mathcal{L}_1 \supseteq \mathcal{OP}$, it is argued that $\boldsymbol{F} \in \mathcal{L}_1$ if $\boldsymbol{F} \in \mathcal{OP}$ by induction on the number of vertices in $\boldsymbol{F}$.
		If $\boldsymbol{F}$ has at most two vertices, this is clear. Suppose $\boldsymbol{F} = (F, u, v)$ has $n \geq 3$ vertices. By \cref{gen2,gen3} of \cref{lem:generation} and the closure properties of $\mathcal{L}_1$ from \cref{def:l-lplus}, it may be supposed that $u \neq v$ and $uv \not\in E(F)$. In this case, again by \cref{lem:generation}, $\boldsymbol{F} = \boldsymbol{K} \cdot \boldsymbol{L}$ for graphs $\boldsymbol{K}$ and $\boldsymbol{L}$, to which the inductive hypothesis applies. 
        It follows that $\boldsymbol{F} \in \mathcal{L}_1$.
	\end{proof}

	\section{Deciding Exact Feasibility of the Lasserre Relaxation with Non-Negativity Constraints in Polynomial Time}
    \label{sec:polytime}

    This section is dedicated to proving \cref{thm:main5}.
    To that end, it is argued that $\simeq_t^{\textup{L}^+}$ has equivalent characterisations in terms of a counting logic and a colouring algorithm akin to the Weisfeiler--Leman algorithm~\cite{weisfeiler_construction_1976}.
    This algorithm has polynomial running time. It is defined as follows:
	
	\begin{definition} \label{def:alg-mwl}
		Let $t \geq 1$. For a graph $G$, an integer $i \geq 1$, and $\boldsymbol{r}, \boldsymbol{s} \in V(G)^t$, define
		\begin{align*}
			\mathsf{mwl}_G^0(\boldsymbol{r}\boldsymbol{s}) &\coloneqq \rel_G(\boldsymbol{r}\boldsymbol{s}), \\
			\mathsf{mwl}_G^{i - 1/2}(\boldsymbol{r}\boldsymbol{s}) &\coloneqq \left( \mathsf{mwl}_G^{i -1}(\sigma(\boldsymbol{r}\boldsymbol{s})) \ \middle| \ \sigma \in \mathfrak{S}_{2t} \right),\\
			\mathsf{mwl}_G^{i}(\boldsymbol{r}\boldsymbol{s}) &\coloneqq \left( \mathsf{mwl}_G^{i - 1 / 2}(\boldsymbol{r}\boldsymbol{s}), \multiset{ 
				\left( \mathsf{mwl}_G^{i - 1 / 2}(\boldsymbol{r}\boldsymbol{t}), \mathsf{mwl}_G^{i - 1 / 2}(\boldsymbol{t}\boldsymbol{s}) \right) \ \middle|\ \boldsymbol{t} \in V(G)^t 
			} \right).
		\end{align*}
		The $\mathsf{mwl}_G^{i}$ for $i \in \mathbb{N}$ define increasingly fine colourings of $V(G)^{2t}$.
		Let $\mathsf{mwl}_G^{\infty}$ denote the finest such colouring.
		Two graphs $G$ and $H$ are not \emph{distinguished by the $t$-dimensional $\mathsf{mwl}$ algorithm} if the multisets 
		\[
			\multiset{\mathsf{mwl}_G^{\infty}(\boldsymbol{r}\boldsymbol{s}) \ \middle|\ \boldsymbol{r}, \boldsymbol{s} \in V(G)^t} \quad \text{ and } \quad
			\multiset{\mathsf{mwl}_H^{\infty}(\boldsymbol{u}\boldsymbol{v}) \ \middle|\ \boldsymbol{u}, \boldsymbol{v} \in V(H)^t}
		\]
		are the same.
	\end{definition}

	Since the finest colouring $\mathsf{mwl}_G^{\infty}$ is reached in $\leq n^{2t} - 1$ iterations for graphs on $n$~vertices, for fixed $t$, it can be tested in polynomial time whether two graphs are not distinguished by the $t$-dimensional $\mathsf{mwl}$ algorithm. We are about to show that the latter happens if and only if the level-$t$ Lasserre relaxation with non-negative constraints is feasible. As a by-product, we obtain a logical characterisation for this equivalence relation akin to \cite{cai_optimal_1992}.
    	
	\begin{definition} \label{def:logicmt}
		For $t \geq 1$, 
        an \emph{$\mathsf{M}^t$-formula} has $2t$ free variables and is of the following form:
        \begin{itemize}
            \item every quantifier-free $\mathsf{FO}$-formula with equality over the signature $\{E\}$ with $2t$ variables is an $\mathsf{M}^t$-formula,
            \item if $\phi,\psi$ are $\mathsf{M}^t$-formulae with the same free variables then $\neg \phi$, $\phi \land \psi$, and $\phi \lor \psi$ are $\mathsf{M}^t$-formulae,
            \item if $\phi, \psi$ are $\mathsf{M}^t$-formulae and $n \in \mathbb{N}$ then $\exists^{\geq n} \boldsymbol{y}.\left( \phi(\boldsymbol{x},\boldsymbol{y}) \land \psi(\boldsymbol{y},\boldsymbol{z}) \right)$ is an $\mathsf{M}^t$-formula. Here, the boldface letters $\boldsymbol{x}$, $\boldsymbol{y}$, $\boldsymbol{z}$ denote pairwise disjoint $t$-tuples of distinct variables.
        \end{itemize}
        An \emph{$\mathsf{M}^t$-sentence} is an expression $\exists^{\geq n} \boldsymbol{x}.\ \phi(\boldsymbol{x})$ where $\phi$ is an $\mathsf{M}^t$-formula, $\boldsymbol{x}$ is a tuple of $2t$ distinct variables, and $n \in \mathbb{N}$.
	\end{definition}

    The semantics of the quantifier $\exists^{\geq n} \boldsymbol{y}.\ \phi(\boldsymbol{y})$ is that there exist at least $n$ many $\lvert \boldsymbol{y} \rvert$-tuples of vertices from the graph over which the formula is evaluated which satisfy $\phi$.
    The following \cref{thm:l+logic-global} may be thought of as a analogue of \cref{thm:sa} for $\mathcal{L}_t^+$.

	\begin{theorem} \label{thm:l+logic-global}
		Let $t \geq 1$.
		For graphs $G$ and $H$, the following are equivalent:
		\begin{enumerate}
			\item $G$ and $H$ are not distinguished by the $t$-dimensional $\mathsf{mwl}$ algorithm,\label{cor:l+logic1}
			\item $G$ and $H$ are homomorphism indistinguishable over $\mathcal{L}_t^+$,\label{cor:l+logic2}
			\item $G$ and $H$ satisfy the same $\mathsf{M}^t$-sentences.\label{cor:l+logic3}
		\end{enumerate}
	\end{theorem}

    The proof of \cref{thm:l+logic-global} is conceptually similar to arguments of \cite{cai_optimal_1992,dvorak_recognizing_2010}.
    It is implied by the following \cref{thm:l+logic}:

    \begin{theorem} \label{thm:l+logic}
		Let $t \geq 1$.
		For graphs $G$ and $H$ with $\boldsymbol{r},\boldsymbol{s} \in V(G)^t$ and $\boldsymbol{u},\boldsymbol{v} \in V(H)^t$, the following are equivalent:
		\begin{enumerate}
			\item $\mathsf{mwl}_G^\infty(\boldsymbol{r}\boldsymbol{s}) = \mathsf{mwl}_H^\infty(\boldsymbol{u}\boldsymbol{v})$,\label{thm:l+logic1}
			\item $\boldsymbol{F}_G(\boldsymbol{r}\boldsymbol{s}) = \boldsymbol{F}_H(\boldsymbol{u}\boldsymbol{v})$ for all $\boldsymbol{F} \in \mathcal{L}_t^+$, and\label{thm:l+logic2}
			\item $G \models \phi(\boldsymbol{r}\boldsymbol{s})$ if and only if $H \models \phi(\boldsymbol{u}\boldsymbol{v})$ for all $\mathsf{M}^t$-formulae $\phi$.\label{thm:l+logic3}
		\end{enumerate}
	\end{theorem}
	
	The proof of \cref{thm:l+logic} is based on the following lemma, which is adopted from \cite[Lemma~6]{dvorak_recognizing_2010}.
	An \emph{$\mathcal{L}_t^+$-quantum graph} is a finite linear combination $q = \sum \alpha_i \boldsymbol{F}^i$ of graphs $\boldsymbol{F}^i \in \mathcal{L}_t^+$ with real coefficients $\alpha_i \in \mathbb{R}$.
	Operations like series or parallel composition can be extended linearly to quantum graphs. Write $q_G$ for the linear combination $\sum \alpha_i \boldsymbol{F}^i_G$ of homomorphism tensors.

	\begin{lemma} \label{lem:dvo6}
		Let $t \geq 1$ and $n \in \mathbb{N}$.
		For every $\mathsf{M}^t$-formula $\phi$,
		there exists an $\mathcal{L}_t^+$-quantum graph $q$ such that for all graphs $G$ on at most $n$ vertices and $\boldsymbol{r},\boldsymbol{s} \in V(G)^t$,
		\begin{itemize}
			\item if $ G \models \phi(\boldsymbol{rs})$ then $q_G(\boldsymbol{r}\boldsymbol{s}) = 1$, and
			\item if $ G \not\models \phi(\boldsymbol{rs})$ then $q_G(\boldsymbol{r}\boldsymbol{s}) = 0$.
		\end{itemize}
	\end{lemma}
	\begin{proof}
		If $\phi$ is a quantifier-free formula then there exists an atomic graph $\boldsymbol{F} \in \mathcal{A}_t$ such that $G \models \phi(\boldsymbol{r}\boldsymbol{s})$ if and only if $\boldsymbol{F}_G(\boldsymbol{r}\boldsymbol{s}) = 1$. Furthermore,  the homomorphism tensor of any atomic graph has entries from $\{0,1\}$.
		
		If $\phi$ is of the form $\neg \psi$, let $q$ denote $\mathcal{L}_t^+$-quantum graph constructed inductively for $\psi$.
		Then $r \coloneqq \boldsymbol{J} - q$ for $\boldsymbol{J}$ as defined in \cref{def:atomic} is an $\mathcal{L}_t^+$-quantum graph as desired.

		If $\phi$ is of the form $\psi \land \chi$ where $\psi$ and $\chi$ have the same free variables as $\phi$, let $q$ and $r$ denote $\mathcal{L}_t^+$-quantum graphs constructed inductively for $\psi$ and $\chi$, respectively.
		Then $s \coloneqq q \odot r$ is an $\mathcal{L}_t^+$-quantum graph as desired.
		
		If $\phi$ is of the form $\psi \lor \chi$ where $\psi$ and $\chi$ have the same free variables as $\phi$, it is equivalent to $\neg(\neg \psi \land \neg \chi)$ and the two previous cases can be applied jointly.
		
		It remains to consider the case in which $\phi$ is of the form $\exists^{\geq \ell} \boldsymbol{y}.\ \psi(\boldsymbol{x},\boldsymbol{y}) \land \chi(\boldsymbol{y},\boldsymbol{z})$.
		Let $q$ and $r$ denote the $\mathcal{L}_t^+$-quantum graphs constructed inductively for $\psi$ and $\chi$, respectively. Then
		$
		(q \cdot r)_G(\boldsymbol{r},\boldsymbol{s}) = \sum_{\boldsymbol{t} \in V(G)^t} q_G(\boldsymbol{r},\boldsymbol{t}) r_G(\boldsymbol{t},\boldsymbol{s})
		$
		is equal to the number of elements $\boldsymbol{t} \in V(G)^t$ such that $G \models \psi(\boldsymbol{r},\boldsymbol{t}) \land \chi(\boldsymbol{t},\boldsymbol{s})$.
		Let $P = \sum c_i x^i \in \mathbb{R}[x]$ be a polynomial which evaluates to $0$ on $\{0,1, \dots, \ell-1\}$ and to $1$ on $\{\ell,\ell+1, \dots, n^t\}$.
		Then the quantum graph $P(q \cdot r) = \sum c_i (q \cdot r)^{\odot i}$ where $(q \cdot r)^{\odot i}$ denotes the parallel composition of $i$~copies of $q \cdot r$ is as desired.
	\end{proof}

	\begin{proof}[Proof of Theorem \ref{thm:l+logic}]
		Supposing \cref{thm:l+logic1}, \cref{thm:l+logic2} is proven by induction on the structure of $\boldsymbol{F}$.
		If $\boldsymbol{F}$ is atomic then the statement follows from $\rel_G(\boldsymbol{r},\boldsymbol{s}) = \rel_H(\boldsymbol{u},\boldsymbol{v})$.
		For $\boldsymbol{F}  = \boldsymbol{K} \odot \boldsymbol{L}$ and $\boldsymbol{F} = \boldsymbol{K}^\sigma$ with $\boldsymbol{K},\boldsymbol{L} \in \mathcal{L}_t^+$, the statement is easily verified.
		It remains to consider the case $\boldsymbol{F}  = \boldsymbol{K} \cdot \boldsymbol{L}$. By definition of $\mathsf{mwl}$, there exists a bijection $\pi \colon V(G)^t \to V(H)^t$ such that 
		\[
			\mathsf{mwl}_G^\infty(\boldsymbol{r}, \boldsymbol{t}) = \mathsf{mwl}_H^\infty(\boldsymbol{u}, \pi(\boldsymbol{t})) \quad \text{and} \quad
			\mathsf{mwl}_G^\infty(\boldsymbol{t}, \boldsymbol{s}) = \mathsf{mwl}_H^\infty(\pi(\boldsymbol{t}), \boldsymbol{v})
		\]
		for all $\boldsymbol{t} \in V(G)^t$. Hence,
		\[
			\boldsymbol{F}_G(\boldsymbol{r}, \boldsymbol{s})
			= \sum_{\boldsymbol{t} \in V(G)^t} \boldsymbol{K}_G(\boldsymbol{r}, \boldsymbol{t}) \boldsymbol{L}_G(\boldsymbol{t}, \boldsymbol{s})
			= \sum_{\boldsymbol{t} \in V(G)^t} \boldsymbol{K}_H(\boldsymbol{u}, \pi(\boldsymbol{t})) \boldsymbol{L}_H(\pi(\boldsymbol{t}), \boldsymbol{v})
			= \boldsymbol{F}_H(\boldsymbol{u}, \boldsymbol{v}).
		\]
		Thus, \cref{thm:l+logic2} holds.
		
		Now suppose that  \cref{thm:l+logic2} holds. 
		If $\phi$ is a $\mathsf{M}^t$-formula such that $G \models \phi(\boldsymbol{r},\boldsymbol{s})$ and $H \not\models \phi(\boldsymbol{u},\boldsymbol{v})$
		then, by \cref{lem:dvo6}, there exists a graph $\boldsymbol{F} \in \mathcal{L}_t^+$ such that $\boldsymbol{F}_G(\boldsymbol{r},\boldsymbol{s}) \neq \boldsymbol{F}_H(\boldsymbol{u},\boldsymbol{v})$. This yields \cref{thm:l+logic3}.
		
		That \cref{thm:l+logic3} implies \cref{thm:l+logic1} is proven similarly as \cite[Theorem~5.2]{cai_optimal_1992} by induction on the number of iterations.
		Since $\rel_G$ and $\rel_H$ can be defined using quantifier-free $\mathsf{M}^t$-formulae, $\mathsf{mwl}_G^0(\boldsymbol{r}, \boldsymbol{s}) = \mathsf{mwl}_H^0(\boldsymbol{u}, \boldsymbol{v})$.
		
		Since $\mathsf{M}^t$ is closed under permuting the names of the variables, it holds for all $\sigma \in \mathfrak{S}_{2t}$ that
		$G \models \phi(\boldsymbol{r}\boldsymbol{s}) \iff   H \models \phi(\boldsymbol{u}\boldsymbol{v})$ for all $\phi \in \mathsf{M}^t$ with $2t$ free variables
		if and only if
		$G \models \phi(\sigma(\boldsymbol{r}\boldsymbol{s})) \iff H \models \phi(\sigma(\boldsymbol{u}\boldsymbol{v}))$ for all $\phi \in \mathsf{M}^t$ with $2t$ free variables.
		Hence, if $\mathsf{mwl}_G^{i}(\boldsymbol{r}\boldsymbol{s}) = \mathsf{mwl}_H^{i}(\boldsymbol{u}\boldsymbol{v})$ for some $i \in \mathbb{N}$ then also
		$\mathsf{mwl}_G^{i+1/2}(\boldsymbol{r}\boldsymbol{s}) = \mathsf{mwl}_H^{i+1/2}(\boldsymbol{u}\boldsymbol{v})$.

        For the step from $i+1/2$ to $i+1$, 
		suppose contrapositively that $\mathsf{mwl}_G^{i+1}(\boldsymbol{r}\boldsymbol{s}) \neq \mathsf{mwl}_H^{i+1}(\boldsymbol{u}\boldsymbol{v})$.
		By the previous argument, it can be supposed that $\mathsf{mwl}_G^{i+1/2}(\boldsymbol{r}\boldsymbol{s}) = \mathsf{mwl}_H^{i+1/2}(\boldsymbol{u}\boldsymbol{v})$.
		Hence, there exists a pair of colours $\left(\mathsf{mwl}_G^{i+1/2}(\boldsymbol{r}\boldsymbol{t}), \mathsf{mwl}_G^{i+1/2}(\boldsymbol{t}\boldsymbol{s})\right)$ which appears in the multisets for $G$ and $H$ differently often, wlog more often in $G$ than in $H$. 
        By the inductive hypothesis, 
        for each pair of distinct $\mathsf{mwl}^{i+1/2}$-colours there exists an $\mathsf{M}^t$-formula $\phi$ which is satisfied by all vertex tuples of the first colour and by none of the second colour. 
        By taking the conjunction of several such formulae, a formula can be constructed which holds for a $2t$-tuple of vertices of $G$ or $H$ if and only if they have a specified colour in $\mathsf{mwl}^{i+1/2}$.
		Let $\phi$ and $\psi$ be formulae which hold exactly for the $2t$-tuples of vertices of $G$ or $H$ of colours $\mathsf{mwl}_G^{i+1/2}(\boldsymbol{r}\boldsymbol{t})$ and $\mathsf{mwl}_G^{i+1/2}(\boldsymbol{t}\boldsymbol{s})$, respectively. 
        By assumption, there is an $N \in \mathbb{N}$ such that the formula $\chi \coloneqq \exists^{\geq N} \boldsymbol{y}.\ \phi(\boldsymbol{x},\boldsymbol{y}) \land \psi(\boldsymbol{y},\boldsymbol{z}) \in \mathsf{M}^t$ is such that $G \models \chi(\boldsymbol{r}\boldsymbol{s})$ and $H \not\models \chi(\boldsymbol{u}\boldsymbol{v})$. This yields \cref{thm:l+logic1}.
	\end{proof}

    Finally, we derive \cref{thm:l+logic-global} from \cref{thm:l+logic}.
	
	\begin{proof}[Proof of Theorem \ref{thm:l+logic-global}]
		Supposing \cref{cor:l+logic1}, let $\pi \colon V(G)^{2t} \to V(H)^{2t}$ be a bijection such that
		$
			\mathsf{mwl}_G^\infty(\boldsymbol{rs}) = \mathsf{mwl}_H^\infty(\pi(\boldsymbol{rs}))
		$
		for all $\boldsymbol{r},\boldsymbol{s} \in V(G)^t$.
		By \cref{thm:l+logic}, for $\boldsymbol{F} = (F, \boldsymbol{u}, \boldsymbol{v}) \in \mathcal{L}_t^+$,
		\[
			\hom(F, G) = 
            \soe(\boldsymbol{F}_G )
            = \sum_{\boldsymbol{r},\boldsymbol{s} \in V(G)^t} \boldsymbol{F}_G(\boldsymbol{r},\boldsymbol{s})
			= \sum_{\boldsymbol{r},\boldsymbol{s} \in V(G)^t} \boldsymbol{F}_H(\pi(\boldsymbol{rs}))
		    = \soe(\boldsymbol{F}_H)
            = \hom(F, H),
		\]
		so \cref{cor:l+logic2} holds.

        Assuming \cref{cor:l+logic2} holds,
		let $\Phi = \exists^{\geq \ell} \boldsymbol{x}.\ \phi(\boldsymbol{x})$ be an $\mathsf{M}^t$-sentence where $\phi$ is an $\mathsf{M}^t$-formula, $\ell \in \mathbb{N}$ and $\boldsymbol{x}$ is a tuple of $2t$ distinct variables.
		Let $q$ denote the $\mathcal{L}_t^+$-quantum graph constructed for $\phi$ and $n \coloneqq \max\{|V(G)|, |V(H)|\}$ via \cref{lem:dvo6}.
		Then \cref{cor:l+logic2} implies that 
        \[
            \left\lvert \{ \boldsymbol{rs} \in V(G)^{2t} \mid G \models \phi(\boldsymbol{rs}) \} \right\rvert 
            = \sum_{\boldsymbol{rs} \in V(G)^{2t}} q_G(\boldsymbol{rs}) 
            = \soe(q_G)
            = \left\lvert \{ \boldsymbol{uv} \in V(H)^{2t} \mid H \models \phi(\boldsymbol{uv}) \} \right\rvert.
        \]
        Hence, $G \models \Phi$ if and only if $H \models \Phi$.
        This yields \cref{cor:l+logic3}.

        Assuming \cref{cor:l+logic3} holds,
		suppose that $G$ and $H$ are distinguished by the $\mathsf{mwl}$ algorithm and let $C \subseteq V(G)^{2t}$ denote an $\mathsf{mwl}$-colour class in $G$ whose counterpart $D \subseteq V(H)^{2t}$ has different size. 
		By \cref{thm:l+logic}, there exists an $\mathsf{M}^t$-formula $\phi$ which is satisfied by tuples in $C$ and $D$ and by no other tuples.
		The $\mathsf{M}^t$-sentence $\exists^{\geq \ell} \boldsymbol{x}.\ \phi(\boldsymbol{x})$ is not satisfied by both $G$ and $H$ for a suitable $\ell \in \mathbb{N}$.
        This yields \cref{cor:l+logic1}.
	\end{proof}

    It would be desirable to extend \cref{thm:main5} to $\simeq_t^{\textup{L}}$. 
    The key property of the graph class $\mathcal{L}_t^+$ which was exploited in the proof of \cref{thm:l+logic-global} is that $\mathcal{L}_t^+$ is closed under arbitrary parallel compositions.
    Therefore, an interpolation argument in the proof of \cref{lem:dvo6} succeeds to reduce testing homomorphism indistinguishability over $\mathcal{L}_t^+$ to the execution of the $\mathsf{mwl}$-colouring  algorithm.
    However, for $\mathcal{L}_t$, arbitrary parallel compositions are not available.
    Thus, designing a suitable colouring algorithm for this graph classes does not seem feasible.

    \section{Conclusion}

    We have established a characterisation of the feasibility of the level-$t$ Lasserre relaxation with and without non-negativity constraints of the integer program $\ISO(G, H)$ for graph isomorphism in terms of homomorphism indistinguishability over the graph classes $\mathcal{L}_t$ and $\mathcal{L}_t^+$. By analysing the treewidth of the graphs $\mathcal{L}_t$ and $\mathcal{L}_t^+$ and invoking results from the theory of homomorphism indistinguishability, we have determined the precise number of Sherali--Adams levels necessary such that their feasibility guarantees the feasibility of the level-$t$ Lasserre relaxation. This concludes a line of research brought forward in \cite{atserias_definable_2023}. For feasibility of the level-$t$ Lasserre relaxation with non-negativity constraints, we have given, besides linear algebraic reformulations generalising the adjacency algebra of a graph, a polynomial time algorithm deciding this property. 

    Missing in \Cref{thm:main} is a tight lower bound on the number of Lasserre levels necessary to ensure feasibility of a given Sherali--Adams level:
    \begin{question} \label{question:conclusion}
        Do there exist for every $t \geq 3$ graphs $G$ and $H$ such that $G \simeq_{t-1}^{\textup{L}} H$ and $G \not\simeq_{t}^{\textup{SA}} H$?
    \end{question}
    Following the path taken in this paper, this question could potentially be resolved in two steps: Firstly, one would need to prove the graph theoretic assertion that the class $\mathcal{L}_t$ does not contain $\mathcal{TW}_t$ for all $t \geq 2$. Secondly, one would need to show that $\mathcal{L}_t$ is homomorphism distinguishing closed or at least that the homomorphism distinguishing closure \cite{roberson_oddomorphisms_2022} of $\mathcal{L}_t$ does not contain $\mathcal{TW}_t$ for all $t \geq 2$. 
    Given the means currently available for proving such a statement \cite{roberson_oddomorphisms_2022,neuen_homomorphism-distinguishing_2023}, 
    this would involve giving game characterisations for $\mathcal{L}_t$ (mimicking the robber-cops game for $\mathcal{TW}_t$) and for $\equiv_{\mathcal{L}_t}$ (similar to the bijective $(t+1)$-pebble game for $\mathcal{TW}_t$). For the former, finding analogies to the notions of brambles or heavens seems necessary~\cite{seymour_graph_1993}.

    Another interesting extension of our work might be an efficient algorithm for computing an explicit partial $t$-equivalence between two graphs, cf.\  \cref{def:pce,def:ce}, or deciding that no such map exists.
    This would yield an efficient algorithm for deciding the exact feasibility of the Lasserre semidefinite program without non-negativity constraints, cf.\  \cite{atserias_definable_2023}.

\newpage

\printbibliography

	\appendix
	\section{Versions of the Lasserre Hierarchy}
	\label{app:lasserre}
	
	The level-$t$ Lasserre relaxation for graph isomorphism studied in \cite{atserias_definable_2023} 
	has a slightly different form.
	For every $\alpha \in \mathbb{N}^{V(G) \times V(H)}$, it comprises a real-valued variable $z_\alpha$.
	For an integer $t \in \mathbb{N}$, write $M_t(z)$ for the matrix whose rows and columns are indexed by $\alpha \in \mathbb{N}^{V(G) \times V(H)}$ such that $|\alpha| \coloneqq \sum_{gh \in V(G) \times V(H)} \alpha_{gh} \leq t$ with $(\alpha, \beta)$-th entry $z_{\alpha + \beta}$. 
	Abusing notation by writing $gh$ for the $gh$-th standard basis vector in $\mathbb{N}^{V(G) \times V(H)}$, the equations can be written as
	\begin{align}
		\label{ao1} M_t(z) &\succeq 0 && \\
		\label{ao2} \sum_{g \in V(G)} z_{\alpha + gh} &= z_\alpha && \text{for all } \alpha \text{ such that } |\alpha| \leq 2t - 2, \\
		\label{ao3} \sum_{h \in V(H)} z_{\alpha + gh} &= z_\alpha && \text{for all } \alpha \text{ such that } |\alpha| \leq 2t - 2, \\
		\label{ao4} z_\alpha &= 0 && \text{for all } \alpha \text{ such that } \alpha \text{ is not a partial isomorphism},\\
		\label{ao5} z_{\alpha + 2gh} &= z_{\alpha +gh}   && \text{for all } \alpha \text{ such that } |\alpha| \leq 2t - 2,\\
		\label{ao6} (z_{\alpha + \beta + gh})_{|\alpha|, |\beta| \leq 2t - 2} & \succeq 0 &&  \text{for all } gh \in V(G) \times V(H), \\
		\label{ao7} (z_{\alpha + \beta} - z_{\alpha + \beta + gh})_{|\alpha|, |\beta| \leq 2t - 2} & \succeq 0 &&\text{for all } gh \in V(G) \times V(H),  \\
		\label{ao8} z_\emptyset &= 1.  &&
	\end{align}

The vector $\alpha$ is a \emph{partial isomorphism} if and only if 
$\rel_G(gg') = \rel_H(hh')$ for all $gh$ and $g'h'$ such that $\alpha_{gh}, \alpha_{g'h'} > 0$.
	
	\begin{lemma} \label{lem:systems}
		The system \cref{ao1,ao1,ao2,ao3,ao4,ao5,ao6,ao7,ao8} has a solution iff 
		the system \cref{lassere1,lassere2,lassere3,lassere4,lassere5} has a solution.
	\end{lemma}
	\begin{proof}
		Let $(y_I)_{I \in \binom{V(G)\times V(H)}{\leq 2t}}$ be a solution to \cref{lassere1,lassere2,lassere3,lassere4,lassere5}.
		For $\alpha \in \mathbb{N}^{V(G) \times V(H)}$ with $|\alpha| \leq 2t$,
		define $z_\alpha$ to be $y_{[\alpha]}$ where $[\alpha] \in \binom{V(G)\times V(H)}{\leq 2t}$ is the set $\{gh \mid \alpha_{gh} \geq 1\}$. Observe that $[\alpha + \beta] = [\alpha] \cup [\beta]$ for all $\alpha, \beta \in \mathbb{N}^{V(G) \times V(H)}$.
		
		By \cref{lassere1}, let $v_I$ for $I \in \binom{V(G)\times V(H)}{\leq t}$ be real vectors such that $y_{I \cup J} = \left< v_I, v_J \right>$ for all $I, J \in \binom{V(G)\times V(H)}{\leq 2t}$. Then
		$
			z_{\alpha + \beta} = y_{[\alpha + \beta]} = y_{[\alpha] \cup [\beta]} = \langle v_{[\alpha]}, v_{[\beta]} \rangle.
		$
		Hence, the vectors $v_{[\alpha]}$ for $\alpha \in \mathbb{N}^{V(G) \times V(H)}$ with $|\alpha| \leq 2t$ witness \cref{ao1}.
		Furthermore,
		\[
			z_{\alpha + \beta + gh} = y_{[\alpha + \beta + gh]} = y_{[\alpha] \cup \{gh \} \cup [\beta] \cup \{gh\}} = \langle v_{[\alpha + gh]}, v_{[\beta +gh]} \rangle.
		\]
		Hence, the matrix in \cref{ao6} is positive semidefinite. For \cref{ao7}, similarly,
		\[
		z_{\alpha + \beta} - z_{\alpha + \beta + gh} 
		= z_{\alpha + \beta} + z_{\alpha + \beta + gh} - z_{\alpha + \beta + gh} - z_{\alpha + \beta + gh}
		= \langle v_{[\alpha]} - v_{[\alpha +gh]}, v_{[\beta]} - v_{[\beta +gh]} \rangle.
		\]
		\Cref{ao2,ao3,ao4,ao8} follow from \cref{lassere2,lassere3,lassere4,lassere5}, respectively.
		\Cref{ao5} holds by definition since $[\alpha + 2gh] = [\alpha +gh]$.

		Conversely, let $(z_\alpha)$ be a solution to \cref{ao1,ao1,ao2,ao3,ao4,ao5,ao6,ao7,ao8}.
		Define $y_I$ for $I \in \binom{V(G)\times V(H)}{\leq 2t}$ as $z_{\delta_I}$ where $\delta_I \in \mathbb{N}^{V(G) \times V(H)}$ is the indicator vector of $I$, i.e.\  $(\delta_I)_{gh} = 1$ if $gh \in I$ and $(\delta_I)_{gh} = 0$ otherwise.
		By \cref{ao5}, $z_{\delta_{I \cup J}} = z_{\delta_I + \delta_J}$ for all $I, J \in \binom{V(G)\times V(H)}{\leq 2t}$.
		Hence, \cref{lassere2,lassere3,lassere4,lassere5} follow from \Cref{ao2,ao3,ao4,ao8}, respectively.
		
		For \cref{lassere1}, let $v_\alpha$ for $\alpha \in \mathbb{N}^{V(G) \times V(H)}$ with $|\alpha| \leq t$ be vectors such that $z_{\alpha+\beta} = \left< v_\alpha, v_\beta \right>$ for all $\alpha, \beta$, by \cref{ao1}. Then
		$
			y_{I \cup J} = z_{\delta_I + \delta_J} = \langle v_{\delta_I}, v_{\delta_J} \rangle.
		$ for all $I, J \in \binom{V(G)\times V(H)}{\leq 2t}$. Hence, \cref{lassere1} holds.
	\end{proof}

\end{document}